\documentclass[11pt,leqno]{article}%
\usepackage{amssymb,bbm,amsmath,
amsthm,
amscd}
\usepackage{vmargin}
\usepackage[dvipsnames]{xcolor}
\usepackage{
 array}
\catcode`\@=11
\newbox\bz@
\newdimen\bdimz@
\def\linethrough#1{\setbox\bz@=\hbox{#1}%
\bdimz@=\ht\bz@ \divide\bdimz@ by 5 \advance\bdimz@ by -\dp\bz@ \ht\bz@=\bdimz@
\leavevmode\hbox{$\overline{\overline{\box\bz@}}$\relax}}
\catcode`\@=12

\def \new #1 {\textcolor{Maroon}{\underline{#1}} }

\usepackage[colorlinks=true]{hyperref}
\newcommand{\BEQ} {\begin{equation} }
  \newcommand{\EEQ} {\end{equation} }
\newcommand{\fin}{\hfill $\Box$}
\newcommand{\BTHM}{\begin{theorem}}
  \newcommand{\ETHM}{\end{theorem}}
\hypersetup{urlcolor=blue, citecolor=red}

\newtheorem{lemma}{Lemma}[section]

\newtheorem{theorem}{Theorem}[section]
\newtheorem{definition}{Definition}[section]
\newtheorem{proposition}{Proposition}[section]
\newtheorem{remark}{Remark}[section]

\newcommand{\Om}{\Omega}
\newcommand{\tore}{\mathbb{T}_3}

\newcommand{\expk}{e^{-i\mathbf{k \cdot x}} } 

\newcommand{ \E}{\varepsilon}
\newcommand{ \vit}{\hbox{\bf u}}

\newcommand{ \vittest }{\hbox{\bf v}}
\newcommand{ \wit}{\hbox{\bf w}}
\newcommand{ \bk}{\hbox{\bf k}}
\newcommand{ \bu}{\hbox{\bf u}}

\newcommand{ \bn}{D_N}
\newcommand{ \obu}{\overline{\bu}}
\newcommand{ \bw}{\hbox{\bf w}}

\newcommand{ \obw}{\overline{\bw}}
\newcommand{ \bH}{\mathbf{H}}

\newcommand{ \bff}{\hbox{\bf f}}
\newcommand{ \Z}{\mathbbm{Z}}
\newcommand{\moy} {\overline {\vit} }
\newcommand{\g} {\nabla }
\newcommand{\p} {\partial}
\newcommand{\x} {{\bf x}}
\newcommand{\N}{\mathbb{N}}
\newcommand{\R}{\mathbb{R}}
\newcommand{\ind}{m}
\newcommand{\witt}{\overline \wit}

\makeatletter 

\@addtoreset{equation}{section}

\begin{document}

\title{\bf Convergence of approximate deconvolution models to the
  filtered Navier-Stokes Equations}

\author{Luigi C. Berselli \thanks{Dipartimento di Matematica Applicata
    Universit\`a di Pisa, Via F. Buonarroti 1/c, I-56127, ITALY,
    berselli@dma.unipi.it,
    http://www2.ing.unipi.it/\symbol{126}d9378.}  \and Roger
  Lewandowski\thanks{IRMAR, UMR 6625, Universit\'e Rennes 1, Campus
    Beaulieu, 35042 Rennes cedex FRANCE;
    Roger.Lewandowski@univ-rennes1.fr,
    http://perso.univ-rennes1.fr/roger.lewandowski/}} \date{}
\maketitle

\begin{abstract}
  We consider a 3D~Approximate Deconvolution Model (ADM) which belongs
  to the class of Large Eddy Simulation (LES) models. We work with
  periodic boundary conditions and the filter that is used to average
  the fluid equations is the Helmholtz one. We prove existence and
  uniqueness of what we call a ``regular weak'' solution $(\wit_N,
  q_N)$ to the model, for any fixed order $N\in\N$ of
  deconvolution. Then, we prove that the sequence $\{(\wit_N,
  q_N)\}_{N \in \N}$ converges --in some sense-- to a solution of the
  filtered Navier-Stokes equations, as $N$ goes to infinity.  This
  rigorously shows that the class of ADM models we consider have the
  most meaningful approximation property for averages of solutions of
  the Navier-Stokes equations.
\end{abstract}

MCS Classification : 76D05, 35Q30, 76F65, 76D03

\medskip

\section{Introduction}
It is well-known that the Kolmogorov theory predicts that simulating
turbulent flows by using the Navier-Stokes Equations
\begin{equation}
 \label{eq:nse}
 \begin{aligned}
   \partial_t\bu+\nabla\cdot(\bu\otimes
   \bu)-\nu\Delta\bu+\nabla p=\bff,
   \\
   \nabla\cdot\bu=0,
   \\
   \vit (0, \x) = \vit_0 (\x). 
 \end{aligned}
\end{equation}
requires $\mathcal{N}=O(Re^{9/4})$ degrees of freedom, where $Re=UL
\,\nu^{-1}$ denotes the Reynolds number and $U$ and $L$ are a typical
velocity and length, respectively. This number $\mathcal{N}$ is too
large, in comparison with memory capacities of actual computers, to
perform a Direct Numerical Simulation (DNS). Indeed, for realistic
flows, such as for instance geophysical flows, the Reynolds number is
order $10^{8}$, yielding $\mathcal{N}$ of order $10^{18}$.... This is
why one aims at computing at least the ``mean values'' of the flows
fields, which are the velocity field $\vit=(u^1,u^2,u^3)$ and the
scalar pressure field $p$. This is heuristically motivated from the
fact that some gross characteristics of the flow behave in a more
orderly manner. In the spirit of the work started probably with
Reynolds, this correspond in finding a suitable computational
decomposition
\begin{equation*}
 \bu=\obu+\bu'\qquad\text{and}\qquad p=\overline{p}+p',
\end{equation*}
where the primed variables are fluctuations around the over-lined mean
fields. Fluctuations can be disregarded since generally in
applications knowledge of the mean flow is enough to extract relevant
information on the fluid motion.\par
The ``mean values'' can be defined in several ways (time or space
average, statistical averages...); in particular, if one denotes the
means fields by $\moy$ and $\overline p$, and by assuming that the
averaging operation commutes with differential operators, one gets the
filtered Navier-Stokes equations
\begin{equation}
 \label{eq:nse2}
 \begin{aligned}
   \partial_t\moy+\nabla\cdot(\overline {\bu\otimes
     \bu})-\nu\Delta\moy+\nabla \overline p=\overline \bff,
   \\
   \nabla\cdot\moy=0,
   \\
   \moy (0, \x) = \overline{\vit_0} (\x). 
 \end{aligned}
\end{equation}
This raises the question of the \textit{interior closure problem},
that is the modeling of the tensor $R(\bu) =\overline {\bu\otimes
  \bu}$ in terms of the filtered variables
$(\obu,\overline{p})$. Classical Large Eddy Simulations (LES) models
approximate $R$ by $\wit \otimes \wit - \nu_{T} ( {\bf k} / {\bf k}_c)
\mathcal{D}(\wit)$ where $\wit \approx \moy$, and $\mathcal{D}(\wit) =
(1/2) (\g \wit + \g \wit^T)$. Here $\nu_T$ is an eddy viscosity based
on a ``cut frequency'' ${\bf k}_c$ (see a general setting
in~\cite{Sag2001}). We introduce the new variable $\bw$ since when
using any approximation for $R(\bu)$, one does not write the
differential equation satisfied by $\obu$, but an equation satisfied
by another field $\bw$ which is \textit{hopefully} close enough to
$\bu$. \par
Another way, that avoids eddy viscosities, consists in approaching $R$
by a quadratic term of the form $B(\wit, \wit)$. J.~Leray~\cite{JL34}
already used in the 1930s the approximation (with our LES notation)
$B(\wit, \wit) = \overline \wit \otimes \wit $ to get smooth
approximations to the Navier-Stokes Equations. This approximations
yield the recent Leray-alpha fashion models, considered to be LES
models, and a broad class of close models (see
\textit{e.g.}~\cite{CHOT05,FHT01,BIL2006,Geu2003,LMC2005}). In
\cite{LL03,LL2006a}, we also have studied the model defined by
$B(\wit, \wit) = \overline { \wit \otimes \wit}$, which has also
strict connections with scale-similarity models.\par
The model we study in this paper, is the Approximate Deconvolution
Model (ADM), first introduced by Adams and Stolz~\cite{SA99,AS2001},
so far as we know. This model is defined by
$B(\wit,\wit)=\overline{D_N(\wit)\otimes D_N(\wit)}$.  Here the
operator $G$ is defined thanks to the Helmholtz filter
(cf.~(\ref{eq:differential_filter})-(\ref{eq:TTTOB}) below) by
$G(\vittest)=\overline{\vittest}$, where in the paper $G=
(\mathrm{I}-\alpha^2 \Delta)^{-1}$, and the operator $D_N$ has the
form $D_N=\sum_{n=0}^N(\mathrm{I} - G)^n$. Therefore, the initial
value problem we consider is:
\begin{equation} \label{eq:adm1}
  \begin{aligned}
   \partial_t \wit +\nabla\cdot(\overline {D_N(\wit) \otimes
     D_N(\wit)})-\nu\Delta\wit+\nabla q= \overline \bff,   
   \\
   \nabla\cdot\wit =0,   
   \\
   \wit (0, \x) = \overline{\vit_0} (\x), 
 \end{aligned}
\end{equation}
and we are working with periodic boundary conditions. We already
observed that the equations~\eqref{eq:adm1} are not the
equations~\eqref{eq:nse2} satisfied by $\obu$, but we are aimed at
considering~\eqref{eq:adm1} as an approximation
of~\eqref{eq:nse2}. This is mathematically sound since formally
$D_N\to A:=\mathrm{I}-\alpha^2 \Delta$, in the limit $N\to+\infty$,
hence again formally~\eqref{eq:adm1} will become the filtered
Navier-Stokes~\eqref{eq:adm2}. What is more challenging is to
understand whether this property is true or not, in the sense that one
would like to show that as the approximation parameter $N$ grows, then
\begin{equation*}
 \bw\to\obu\qquad\text{and}\qquad q\to\overline{p}.
\end{equation*}
One would like to prove that solutions of the model converge to
averages of the true quantities, since, we recall that the main goal
of LES as a computational tool is {to approximate the averages of the
  flow, which are the only interesting and computable quantities.}  To
this end we want to point out that, beside the technical mathematical
difficulties, proving results of approximation of single trajectories
\begin{equation*}
 \bw\to\bu\qquad \text{and}\qquad q\to{p},
\end{equation*}
is not really in the ``rules of the game,'' because the consistency of
the model towards single (generally strong or not computable)
solutions to the Navier-Stokes equations is not the most interesting
point. Anyway, such convergence is generally known only for a few
models, as for instance for many of the alpha-models, as
$\alpha\to0^+$. To support again this point of view, observe that
generally $\alpha>0$ is related to the smallest resolved/resolvable
scale, hence $\alpha=\mathcal{O}(h)$, where $h$ is the mesh
size. Reducing $\alpha$ will mean resolving completely the flow, hence
performing a DNS instead of a LES.\par
To our knowledge such a ``\textit{well posedness},'' i.e. proving that
$\bw\to\obu$, is not known for any LES model: To our knowledge there
are no results showing (not only formally, but also rigorously) that
the solution of a LES model $(\bw,q)$ is close or converges in some
sense to the \textit{averages} $(\obu,\overline{p})$. \par
To continue the introduction to our new results, we recall that
model~\eqref{eq:adm1} has already been considered in~\cite{LL2006b}
where we studied the residual stress. It has also been studied in
Dunca and Epshteyn~\cite{DE2006}, where it has been proved the
existence of a unique ``smooth enough'' solution for periodic boundary
conditions. In~\cite{DE2006} it is also shown that, that the sequence
of models~(\ref{eq:adm1}) goes -in a certain meaning- to the
Navier-Stokes Equations when $\alpha\to0^+$, for $N$ fixed. Notice
that the model intensively investigated
in~\cite{LL03,LL2006a,CLT2006}, where $B(\wit, \wit) = \overline {
  \wit \otimes \wit}$, is the special case $N=0$ and it is also called
simplified-Bardina, since it resembles some features of the scale
similarity models.\par
The main topic of this paper is then to study what happens when $N$
goes to infinity in~\eqref{eq:adm1}. We prove that the sequence of
models~(\ref{eq:adm1}) converges, in some sense, to the averaged
Navier-Stokes equations~\eqref{eq:nse2}, when the typical scale of
filtration (called $\alpha>0$) remains fixed and the boundary
conditions are the periodic ones. Before analyzing such convergence we
need to prove existence of smooth enough solutions. To this end we
needed to completely revisit the approach in~\cite{DE2006}.  To be
more precise, let $\mathbb{T}^3$ be the $3D$ torus and let $(\wit_N,
q_N)$, with 
\begin{equation*}
 \begin{aligned}
   &\wit_N\in L^2 ([0,T]; H^2(\tore)^3) \cap L^\infty ([0,T];
   H^1(\tore)^3), 
   \\
   &q_N\in L^2([0,T];W^{1,2}(\tore)\cap L^{5/3}([0,T];  W^{2, 5/3}(\tore)),
 \end{aligned}
\end{equation*}
denote the solution of the ADM model~(\ref{eq:adm1}), where $T$ is a
fixed time, that can eventually be taken to be equal to $\infty$,
assuming ${\bf f} \in L^2([0,T]; (H^1(\tore)^3)')$ and $\vit_0 \in
L^2(\tore)^3$, an assumption that we do throughout the paper. We are
able to prove existence (cf. Theorem~\ref{thm:existence}) in such
class and our main result is the following.
\begin{theorem}~\label{thm:Principal} From the sequence $\{(\wit_N,
  q_N)\}_{N \in \N}$ one can extract a sub-sequence (still denoted
  $\{(\wit_N, q_N)\}_{N \in \N}$) such that
 \begin{equation*}
   \begin{aligned}
     &\bw_N\to\bw\quad&\text{weakly in }L^2 ([0,T]; H^2(\tore)^3) \cap
     L^\infty ([0,T]; H^1(\tore)^3),
     \\
     &\bw_N\to\bw&\text{strongly in } L^p ([0,T];
     H^1(\tore)^3),\quad\forall \,1\leq p<+\infty,\qquad
     \\
     &q_N\to q&\text{weakly in }L^2([0,T];W^{1,2}(\tore)\cap
     L^{5/3}([0,T]; W^{2,5/3}(\tore)),
   \end{aligned}
 \end{equation*}
 and such that the system
 \begin{equation}
   \label{eq:adm2}
   \begin{aligned}
     \partial_t \wit +\nabla\cdot(\overline {A\wit \otimes
       A\wit})-\nu\Delta\wit+\nabla q= \overline \bff,
     \\
     \nabla\cdot\wit =0,
     \\
     \wit (0, \x) = \overline{\vit_0} (\x),
   \end{aligned}
 \end{equation}
 holds in the distributional sense, {where we recall that $A =
   G^{-1}=\textrm{I}-\alpha^2\Delta$}. Moreover, the following energy
 inequality holds:
 \begin{equation}
   \frac{1}{2}{d \over  dt}\| A \wit \|^2 + 
   \nu \| \g A \wit \|^2 \le \langle{\bf f}, A \wit\rangle.
 \end{equation}Ê
\end{theorem}
As a consequence of Theorem \ref{thm:Principal}, we deduce that the
field $(\vit, p)=(A\wit, A q)$ is a dissipative (Leray-Hopf) solution
to the Navier-Stokes Equations~\eqref{eq:nse}.
\begin{remark}
  If one rewrites system~\eqref{eq:nse2} in terms of the variables
  $\bw=\obu$ and $\bu=A\obu=A\bw$, one obtains exactly the
  system~\eqref{eq:adm2}. This is not a LES model, since it is just a
  change of variables. The LES modeling comes into the equations with
  the approximation of the operator $A$ by means of the family
  $\{D_N\}_{N\in\N}$.
\end{remark}
\noindent\textbf{Plan of the paper.} Since the paper deals mainly with
the mathematical properties of the model, we start in
Section~\ref{sec:background} by giving a precise definition of our
filter through the Helmholtz equation and we sketch a reminder of the
basic properties of the deconvolution operator $D_N$. The precise
knowledge of the filter is one of the critical points in the analysis
we will perform. We also claim that, beside some basic knowledge of
functional analysis, we have been able to simplify the proof in order
to employ just the classical energy and compactness methods. Roughly
speaking, we needed to find the correct \textit{multipliers} and --at
least in principle-- the proof of the main result should be readable
also from practitioners.\par
Then, we show in Section~\ref{sec:existence} an existence and
uniqueness result for system~(\ref{eq:adm1}). Even if this result has
already been obtained by Dunca and Epshteyn~\cite{DE2006}, our proof
is shorter and uses different arguments, useful for proving our main
convergence result. Indeed, Dunca and Epshteyn proved initially a
smart but very technical formula about $D_N$ in terms of series of
$(-\Delta)^k$, but they did not get uniform estimates in $N$. This is
why their proof cannot help for passing to the limit when $N$ goes to
infinity. Our first main observation in this paper is that one can get
very easily an estimate uniform in $N$ for $A^{1/2} D_N^{1/2} (\wit)$,
that also yields estimates for $D_N(\wit)$ and $\wit$ itself, always
uniform in $N$. The \textit{leitmotiv} of the paper is to prove
estimates \textit{independent} of $N$. \par
Finally, Section~\ref{sec:limit} is devoted to the proof of
Theorem~\ref{thm:Principal}. To prove it, we use the estimates,
uniform in $N$, obtained in the proof of the existence result. Another
main ingredient of the proof, is the derivation of an estimate for
$\p_t D_N (\wit_N)$. Good estimates for this term yield a compactness
property (\`a la Aubin-Lions) for $D_N (\wit_N)$, which allows us to
pass to the limit in the non linear term. We also note that in our
argument we keep control of the pressure, since it is needed in some
arguments and we do not simply neglect it by projecting the equations
over divergence-free vector fields.

\medskip

{\bf Acknowledgements}. The work of Roger Lewandowski is partially
supported by the ANR project 08FA300-01. Roger Lewandowski also warmly
thanks the department of applied math. of the University of Pisa for
the hospitality, where part of the work has been done. Luigi
C. Berselli gratefully acknowledges the hospitality of IRMAR,
Universit\'e Rennes 1, where part of the work has been done.

\section{General Background}\label{sec:background}
\subsection{Orientation} 
This section is devoted first to the definition of the function spaces
that we use, next to the definition of the filter through the
Helmholtz equation, and finally to what we call the ``deconvolution
operator.'' There is nothing new here that is not already proved in
former papers. This is why we restrict our-selves to what we need for
our display and we skip out proofs and technical details. Those
details can be proved by standard analysis and the reader can check
them in several references already quoted in the introduction and also
quoted below in the text.
\subsection{Function spaces}
In the sequel we will use the customary Lebesgue $L^p$ and Sobolev
$W^{k,p}$ and $W^{s,2}=H^s$ spaces. Since we work with periodic
boundary conditions we can better characterize the divergence-free
spaces we need. In fact, the spaces we consider are well-defined by
using Fourier Series on the 3D torus $\tore$ defined just below.  Let
be given $L \in \R^\star_+=\{x\in\R:\ x>0\}$, and define $\Omega :=
[0,L]^3 \subset \R^3$. We denote by $({\bf e}_1, {\bf e}_2, {\bf
  e}_3)$ the orthonormal basis of $\R^3$, and by $\x=(x_1, x_2, x_3)
\in \R^3$ the standard point in $\R^3$. We put ${\cal T}_3 := 2 \pi
\Z^3 /L $. Let $\mathbb{T}_3$ be the torus defined by $\mathbb{T}_3=
\left ( \R^3 / {\cal T}_3 \right )$.  We use $\|\cdot\|$ to denote the
$L^{2}(\mathbb{T}_3)$\ norm and associated operator norms. We always
impose the zero mean condition $\displaystyle \int_{\Omega}\phi
\,d\x=0$ on $\phi=\wit,p,\mathbf{f}$, or $\wit_{0}$.
We define, for a general exponent $s\geq0$, 
\begin{equation*}
 \mathbf{H}_{s} = \left\{\wit : \tore \rightarrow \R^3, \, \,
   \mathbf{w} \in H^{s}(\mathbb{T}_3)^3,   \quad\nabla\cdot\mathbf{w}
   = 0, \quad\int_{\mathbb{T}_3}\mathbf{w}\,d\x = \mathbf{0} \right\},
\end{equation*}
where $H^{s}(\mathbb{T}_3)^k=\big[H^{s}(\mathbb{T}_3)\big]^k$, for all
$k\in\N$ (If $0\leq s<1$ the condition $\nabla\cdot \bw=0$ 
must   be
understood in a weak sense).\par
For $\wit \in {\bf H}_s$, we can expand the velocity field in a
Fourier series
\begin{equation*}
 \wit (\mathbf{x})=\sum_{\mathbf{k} \in {\cal T}_3^\star}\widehat{\wit}_{\bf k}
 e^{+i\mathbf{k \cdot x}},\text{ where
 }\mathbf{k}\in{\cal T}_3^\star\text{ is the wave-number,}
\end{equation*}
and the Fourier coefficients are given by
\begin{equation*}
\widehat{\wit}_{\bf k}=\frac{1}{|\mathbb{T}_3 | }\int_{\mathbb{T}_3}%
\wit(\x)e^{-i\mathbf{k  \cdot x}}d\mathbf{x.}%
\end{equation*}
The magnitude of $\mathbf{k}$ is defined by
\begin{align*}
k:=&|\mathbf{k}|=\{|k_{1}|^{2}+|k_{2}|^{2}+|k_{3}|\}^{\frac{1}{2}}.
\end{align*}
We define the $\mathbf{H}_{s}$ norms by
\begin{equation*}
 \| \wit \|^{2}_{s} = \sum_{\mathbf{k} \in {\cal T}_3^\star} | \mathbf{k} |^{2s} |\widehat
 {\wit }_{\bf k}|^{2},
\end{equation*}
where of course $\| \wit \|^{2}_{0} = \| \wit \|^{2}$. The inner
products associated to these norms are
\begin{equation}\label{TQW887} 
  (\wit, \hbox{\bf v} )_{\mathbf{H}_s} = \sum_{\mathbf{k} \in {\cal T}_3^\star} 
  | \mathbf{k} |^{2s}  \widehat
  {\wit }_{\bf k}\cdot  \overline{\widehat
    {\hbox{\bf v} }_{\bf k}},
\end{equation}
where here, and without risk of confusion with the filter defined
later, $\overline{\widehat{\hbox{\bf v} }_{\bf k}}$ denotes the
complex conjugate of $\widehat{\hbox{\bf v} }_{\bf k}$. This means
that if $\widehat {\bf v}_{\bf k}= (v^1_{\bf k}, v^2_{\bf k}, v^3_{\bf
  k})$, then $\overline{\widehat{\hbox{\bf v} }_{\bf k}} =
(\overline{v^1_{\bf k}}, \overline{v^2_{\bf k}}, \overline{v^3_{\bf
    k}})$.\par
Since we are looking for real valued vectors fields, we have the
natural relation, for any field denoted by $\wit$:
\begin{equation*}
  \widehat{\wit}_{\bf k} = \overline{\widehat{\wit }_{-\bf k}}, 
  \qquad\forall \,   \mathbf{k}\in {\cal T}_3^\star. 
\end{equation*}
Therefore, our space ${\bf H}_s$ is a closed subset of the space
$\mathbb{H}_s$ of complex valued functions
\begin{equation*}
  \mathbb{H}_s = \left \{ \wit  = \sum_{\mathbf{k} \in {\cal
        T}_3^\star} \widehat \wit_{\bf k}e^{+i\mathbf{k \cdot x}}:\ \  
    \sum_{\mathbf{k} \in {\cal T}_3^\star} | \mathbf{k} |^{2s} |\widehat
    {\wit }_{\bf k}|^{2} < \infty,\ \mathbf{k}\cdot \widehat{\bw}_\mathbf{k}=0\right\},
\end{equation*}
equipped with the Hilbert structure given by~(\ref{TQW887}).  It can
be shown (see e.g.~\cite{DG1995}) that when $s$ is an integer,
$\|\wit\|^{2}_{s}=\| \nabla^{s} \wit\|^{2}$. One also can prove 
that for general $s \in \R$, $(\mathbb{H}_s)' = \mathbb{H}_{-s} $  
(see in 
\cite{RLbook}). 
\subsection{About the Filter}
We now recall the main properties of the Helmholtz filter. In the
sequel $\alpha>0$, denotes a given fixed number and $\wit
\in\mathbf{H}_{s}$.  We consider the
Stokes-like problem for $s\geq-1$:
\begin{equation}\label{eq:differential_filter}
\begin{aligned}
  -\alpha^2\Delta\witt +\witt+\nabla \pi=\wit &\qquad\text{in }\tore,
  \\
  \nabla\cdot \witt=0&\qquad\text{in }\tore,
\end{aligned}
\end{equation}
and in addition, $\displaystyle \int_{\tore}\pi \,d\x=0$ to have a
uniquely defined pressure.\par
It is clear that this problem has a unique solution $(\witt,
\pi)\in\mathbf{H}_{s+2} \times H^{s+1} (\tore)$, for any $\wit \in
\mathbf{H}_{s} $. We put $G (\bw) = \obw$, $A= G^{-1}$. Notice that
even if we work with real valued fields, $G= A^{-1}$ maps more
generally $\mathbb{H}_s$ onto $\mathbb{H}_{s+2}$.  Observe also that
-in terms of Fourier series- when one inserts
in~(\ref{eq:differential_filter})
\begin{equation*}
\wit =\sum_{\mathbf{k} \in {\cal T}_3^\star}
\widehat{\wit}_{\bf k}\,e^{+i\mathbf{k \cdot x}},
\end{equation*}
one easily gets, by searching $(\obw,\pi)$ in terms of Fourier Series,
that 
\begin{equation}\label{eq:123LKG}
 \obw(\mathbf{x})=\sum_{\mathbf{k} \in {\cal
     T}_3^\star}{1 \over 1 + \alpha^2 | {\bf k} |^2 }
 \widehat{\wit}_{\bf k}\,e^{+i\mathbf{k \cdot x}} = G(\wit), \qquad
 \text{and}\qquad\pi= \,   0. 
\end{equation}
With a slight abuse of notation,  for a scalar function $\chi$ we still
denote by $\overline \chi$ the solution of the pure Helmholtz problem
\begin{equation}\label{eq:TTTOB} 
 A \overline \chi =-\alpha^2 \Delta \overline \chi + \overline \chi =
 \chi \quad \text{in }\tore, \qquad G(\chi) = \overline \chi.
\end{equation}
and of course there are not vanishing-mean conditions to be imposed
for such cases.  This notation --which is nevertheless historical-- is
motivated from the fact that in the periodic setting and for
divergence-free vector fields the Stokes
filter~\eqref{eq:differential_filter} is exactly the same
as~\eqref{eq:TTTOB}.  Observe in particular that in the LES
model~\eqref{eq:adm1} and in the filtered
equations~\eqref{eq:nse2}-\eqref{eq:adm2}, the symbol
``$\overline{\phantom{a\cdot}}$'' denotes the pure Helmholtz filter,
applied component-by-component to the tensor fields $D_N(\wit) \otimes
D_N(\wit)$, $\bu\otimes\bu$, and $A\bw\otimes A\bw$ respectively.
\subsection{The deconvolution operator}
We start this section with a useful definition, which we shall use
several times in the sequel to understand the relevant properties of
the LES model.
\begin{definition} 
 Let $K$ be an operator acting on ${\bf H}_s$. Assume that $\expk$
 are eigen-vectors of $K$ with corresponding eigenvalues $\widehat
 K_{\bf k}$. Then we shall say that $\widehat K_{\bf k}$ is the
 symbol of $K$.
\end{definition}
The deconvolution operator $D_N$ is constructed thanks to the
Van-Cittert algorithm, and is formally defined by  
\begin{equation}  \label{eq:TTRO}
 D_N := \sum_{n=0}^N ({\rm I}-G)^n.
\end{equation}
The reader will find a complete description and analysis of the
Van-Cittert Algorithm and its variants in~\cite{RL09}. Here we just   report
 the properties we only need for the description of the model.\par
Starting from~(\ref{eq:TTRO}),  we can express the deconvolution
operator in terms of Fourier Series by the formula  
\begin{equation*}
 D_N (\wit) = \sum_{ {\bf k}\in {\cal T}_3^\star}\widehat D_N ({\bf k}) 
 \widehat  \wit_{\bf k} e^{+i\mathbf{k \cdot x}},
\end{equation*}
where
\begin{equation}\label{eq:rep_D_N}
 \begin{aligned}
   \widehat{D}_N({\bf k})=\sum_{n=0}^N\left(\frac{\alpha^2|{\bf
         k}|^2}{1+\alpha^2|{\bf k}|^2}\right)^n= (1+\alpha ^2 |{\bf
     k}|^2)\rho_{N,{\bf k}} , \quad \rho_{N,{\bf
       k}}=1-\left(\frac{\alpha^2 |{\bf k}|^2}{1+\alpha^2 |{\bf
         k}|^2}\right)^{N+1}.
 \end{aligned}
\end{equation}
The symbol $\widehat{D}_N({\bf k})$ of the operator $D_N$ satisfies the following
\begin{equation}\label{55567I}
 \text{for each } {\bf k} \in {\cal T}_3 \,   \hbox{ fixed} \qquad  \widehat{D}_N({\bf k})\to
 1+\alpha^2|{\bf k}|^2,\quad\text{as }N\to+\infty, 
\end{equation}
even if not uniformly in $\bk$. This means that $\{D_N\}_{N \in \N}$
converges to $A$ in some sense (see Lemma~\ref{lem:converges}
below). We need to specify this convergence in order to pass to the
limit more than in ``a formal way,'' to go from~\eqref{eq:adm1}
to~\eqref{eq:adm2}. A general goal for the all paper is to precisely
determine the notion of $D_N\to A$ and to obtain enough estimates on
the solution $\bw$ of~\eqref{eq:adm1} in order to perform such
limit.\par
The basic properties satisfied by $\widehat{D}_N$ that we will need
are summarized in the following lemma.
\begin{lemma}\label{lower_bound}
 For each $N\in\N$ the operator $\bn:\ \mathbf{H}_{s} \to
 \mathbf{H}_{s}$ is self-adjoint, it commutes
 with differentiation, and the following properties hold true:
 \begin{eqnarray}
   & \label{IINVC9} 1\leq \widehat{D}_N({\bf k})\leq N+1   & \forall\, {\bf k} \in {\cal T}_3, 
   \\
   & \label{JJV34}  \displaystyle \widehat D_N ({\bf k}) \approx (N+1)
   {1+\alpha^2 |{\bf k}|^2 \over \alpha^2 |{\bf k}|^2 } & \hbox{for
     large } |{\bf k}|,  
   \\
   &  \label{TT0ON}  \displaystyle  \lim_{|{\bf k}
     |\to+\infty}\widehat{D}_N({\bf k})=N+1  & \text{for fixed    }\alpha>0,  
   \\  
   & 
   \widehat{D}_N({\bf k})\leq(1+\alpha^2|{\bf k}|^2)  &  \forall\,
   {\bf k} \in {\cal T}_3,\     \alpha>0.   
 \end{eqnarray} 
\end{lemma}
All these claims are ``obvious,'' in the sense that they follow from
direct inspection of the formula~\eqref{eq:rep_D_N}. Nevertheless,
they call for some comments. A first observation is that~(\ref{TT0ON})
is a direct consequence of~(\ref{JJV34}). This says that the
$\mathbf{H}_{s}$ spaces are preserved by the operator $D_N$.  More
precisely, for all $s \ge 0$, the map
\begin{equation*}
     \wit \mapsto D_N (\wit), 
\end{equation*}
is an isomorphism which satisfies 
\begin{equation*}
 \| D_N \|_{\mathbf{H}_s} = O(N+1).
\end{equation*}
Moreover, the term $\overline{D_N (\wit) \otimes D_N (\wit)}$ in
model~(\ref{eq:adm1}) is better than the convective term
$\overline{A\bw\otimes A\bw}$ in the classical filtered Navier-Stokes
Equations, making a good hope for the model to have what we call a
unique ``regular weak'' solution (see Definition~\ref{RegSol} in the
next section) which satisfies an energy equality. This follows because
the sequence $\{D_N\}_{N\in\N}$ is made of differential operators of
zero-order approximating $A$, which is of the second-order.  The
solutions of~\eqref{eq:adm1} are stronger than the usual weak
(dissipative) Leray's solution: this is the good news. The bad news is
that high frequency modes are not under control and may blow up when
one lets $N$ to go to infinity, making very hard the question of the
limit behavior of the sequence of models~(\ref{eq:adm1}), when $N$ goes
to infinity.  \par
In the same spirit of limiting behavior of $D_N$ --as a byproduct
of~(\ref{55567I}) and~(\ref{JJV34})-- it can be shown that the
sequence of operators $\{D_N\}_{N \in \N}$ ``weakly'' converges (more
precisely one has point-wise convergence) to the operator $A$. The
following result holds true.
\begin{lemma}\label{lem:converges}
 Let $s \in \R$ and let $\bw\in \bH_{s+2}$. Then
 \begin{equation*}
   \lim_{N\to+\infty} \bn(\bw)=A\bw\qquad\text{in }\bH_s.
 \end{equation*} 
\end{lemma} 
The proof of this lemma is very close to the one of Lemma~2.5
in~\cite{LL2008}. Therefore, we skip the details.
\begin{remark}\label{rem:fractional}
  Since $D_N$ is self-adjoint and non-negative it is possible to
  define the fractional powers $D^{\alpha}_N$, for $\alpha\geq0$.  From
  the previous result we also obtain directly that if $\bw\in
  \mathbf{H}_s$, then 
\begin{equation*}
  \forall\,\alpha\geq0,\qquad 
  \lim_{N\to+\infty}  A^{-\alpha}D_N^\alpha(\bw)=\bw\qquad \text{in }\mathbf{H}_s.
\end{equation*}
\end{remark}

\begin{remark}
 The reader may observe that most of the properties satisfied by
 $\bn$ are also satisfied by the Yosida approximation
 \begin{equation*}
       A_\lambda:=\frac{\mathrm{I}-(\mathrm{I}+\lambda
         A)^{-1}}{\lambda},\qquad \lambda>0,
 \end{equation*}
 which is very common in the theory of semi-groups or in the calculus
 of variations. To compare the behavior of the two approximations, we
 write the explicit expression for the symbol of the Yosida
 approximation, with $\lambda=1/N$:
 \begin{equation*}
   \widehat{A}_{1/N}({\bf k})=(1+\alpha ^2|{\bf
     k}|^2)\left[1-\frac{1+\alpha^2|{\bf k}|^2}{N+1+\alpha^2|{\bf
         k}|^2}\right]=(1+\alpha^2|{\bf k}|^2)\sigma_{N,{\bf k}},\quad
   \sigma_{N,{\bf k}}=\frac{N}{N+1+\alpha^2|{\bf k}|^2}.
 \end{equation*}
 One can directly compute that the asymptotics are essentially the
 same as in Lemma~\ref{lower_bound}, but the Van Cittert operator
 $\bn$ converges to $A$ much faster than the Yosida approximation
 $A_{1/N}$, as $N$ goes to infinity.
\end{remark}
\section{Existence results} \label{sec:existence}
As we pointed out in the introduction, in this paper we consider the initial
value problem for the LES model~\eqref{eq:adm1}. %
The aim of this section is to prove the existence of a unique solution
to the system~(\ref{eq:adm1}) for a given $N$. In the whole paper,
$\alpha>0$ is fixed as we already have said, and we assume that the
data are such that
\begin{equation}  \label{GGMC0}
 \vit_0 \in \mathbf{H}_{0}, \quad {\bf f}\in  L^2([0,T]; \mathbf{H}_{-1}),
\end{equation}
which naturally yields 
\begin{equation}
 \overline{\vit_0} \in \mathbf{H}_{2}, \quad \overline{\bf
   f}\in L^2([0,T]; \mathbf{H}_{1}).
\end{equation}
We start by defining the notion of what we call a "regular weak"
solution to this system.
\begin{definition}[``Regular weak'' solution]\label{RegSol}
 We say that the couple $(\wit,q)$ is a ``regular weak'' solution to
 system~(\ref{eq:adm1}) if and only if the three following items are
 satisfied: \medskip

 \textcolor{red}{1) \underline{\sc Regularity} }

 \begin{eqnarray}
   && \label{Reg11} \wit \in   L^2([0,T];
   \mathbf{H}_{2}) \cap C([0,T];  \mathbf{H}_{1}),
   \\
   && \label{Reg12} \p_t \wit \in L^2([0,T];  \mathbf{H}_{0}) 
   \\
   && \label{Reg13} q \in L^2([0,T]; H^1(\tore)),
 \end{eqnarray}  
 \medskip

 \textcolor{red}{2) \underline {\sc Initial data} }
 \begin{equation}\label{Initial}
   \displaystyle \lim_{t \rightarrow 0}\|\wit(t, \cdot) -  \overline{\vit_0}  \|_{ \mathbf{H}_{1}} = 0,
 \end{equation}
 \medskip

 \textcolor{red}{3) \underline{\sc Weak Formulation}}

 \begin{eqnarray}
   && \label{Reg14} \forall \, \vittest \in L^2([0,T];
   H^1(\tore)^3), 
   \\ 
   &&  \label{RRQPJ}
   \begin{aligned}
     \displaystyle\int_0^T\int_{\tore} \p_t \wit \cdot \vittest - 
     \int_0^T\int_{\tore} \overline{D_N (\wit) \otimes D_N (\wit)} : \g  \vittest + \nu 
     \int_0^T\int_{\tore} \g \wit : \g  \vittest \,  
     \\ 
     \displaystyle \hskip 6cm
     +\int_0^T\int_{\tore} \g q \cdot \vittest =\int_0^T\int_{\tore} \overline{\bf f} \cdot 
     \vittest.
   \end{aligned}
 \end{eqnarray}  
\end{definition}
Almost all terms in~(\ref{RRQPJ}) are obviously well-defined thanks
to~(\ref{Reg11})-(\ref{Reg12})-(\ref{Reg13}), together
with~(\ref{Reg14}). The convective term however, needs to be checked a
little bit more carefully.  To this end, we first recall that $D_N$
maps $\mathbf{H}_{s}$ onto itself, and the Sobolev embedding implies
that $\wit \in C([0,T]; \mathbf{H}_{1}) \subset L^\infty ([0,T];
L^6(\tore)^3)$.\par
Then, we still have $D_N (\wit) \in C([0,T]; \mathbf{H}_{1}) \subset
L^\infty ([0,T]; L^6(\tore)^3)$. In particular, 
\begin{equation*}
D_N (\wit) \otimes D_N (\wit) \in
L^\infty ([0,T]; L^3(\tore)^3)^2.
\end{equation*}
Consequently, we have at least
\begin{equation*}
\overline{D_N (\wit) \otimes D_N (\wit)} \in L^\infty ([0,T]; H^2(\tore)^3)^2
\subset L^\infty ([0,T] \times \tore)^9,
\end{equation*}
which yields the integrability of $\overline{D_N (\wit) \otimes D_N
 (\wit)} : \g \vittest$ for any $\vittest \in L^2([0,T];
H^1(\tore)^3)$. 
\begin{remark}
  We point out that we use the name ``regular weak'' solution, since
  $(\bw,q)$ is a solution in the sense of distributions, but it will
  turn out to be smooth enough to be uniqueness and to satisfy an
  energy equality in place of only an energy inequality, such as in
  the usual Navier-Stokes equations.
\end{remark}
\begin{theorem}\label{thm:existence}
 Assume that~(\ref{GGMC0}) holds, $\alpha >0$ and $N \in \N$ are
 given and fixed. Then Problem~\eqref{eq:adm1} has a unique regular
 weak solution.
\end{theorem}
\begin{proof}
  We use the usual Galerkin method (see for instance the basics for
  the Navier-Stokes Equations in~\cite{Lio1969}). This allows to
  construct the velocity part of the solution, since the equation is
  projected on a divergence-free vector field space. The pressure is
  recovered by De Rham Theorem at the end of the process, that we
  divide into five steps:
  \begin{enumerate}
  \item[]{\sc Step 1}: we start by constructing approximate solutions
    $\wit_\ind $, solving differential equations on finite dimensional
    spaces (see Definition~\ref{TYUKA28} below);
  \item[]{\sc Step 2}: we look for bounds on $\{\wit_\ind\}_{\ind \in
      \N}$ and $\{\p_t \wit_\ind\}_{\ind \in \N}$, uniform with
    respect to $\ind\in \N$, in suitable spaces. This is obtained by
    using an energy-type equality satisfied by $A^{1/2} D_N^{1/2}(
    \wit_\ind)$. Most of these bounds will result also \textit{uniform
      in $N$}, where $N\in\N$ is the index related to the order of
    deconvolution of the model;
  \item[]{\sc Step 3}: we use the main rules of functional analysis to
    get compactness properties about the sequence $\{\wit_\ind\}_{n
      \in \N}$. This will allow us to pass to the limit when $\ind\to
    \infty$ and $N$ is fixed, to obtain a solution to the model;
  \item[] {\sc Step 4}: we check the question of the initial data;
  \item[] {\sc Step 5}: we show that the solution we constructed is
    unique thanks Gronwall's lemma.
  \end{enumerate}
  Since Step~1 and~3 are very classical, we will only sketch them, as
  well as Step~4 which is very close from what has already been done
  in~\cite{CF1988,RLbook,Tem2001}. On the other hand, Step~2 is one of
  the main original contributions in the paper and will also be useful
  in the next section. Indeed, we obtain many estimates, uniform in
  $N$, that allow us in passing to the limit when $N$ goes to
  infinity and proving Theorem~\ref{thm:Principal}. Also Step~4 needs
  some application of classical tools in a way that is less standard
  than usual. We also point out that Theorem~\ref{thm:existence}
  greatly improves the corresponding existence result in~\cite{DE2006}
  and it is not a simple restatement of those results.

  \bigskip

  \noindent \textcolor{blue}{\underline{\sc Step 1 : construction of
      velocity's approximations}.}

  \medskip

  Let be given $\ind \in \N^\star$ and define ${\bf V}_\ind$ to be the
  space of real valued trigonometric polynomial vector fields of
  degree less or equal than $n$, with vanishing both divergence and
  mean value on the torus $\tore$,
  \begin{equation}\label{TYUKA28}
    {\bf V}_\ind := \{ \wit \in  \mathbf{H}_{1}:\quad \int_{\tore}\wit (\x) \, \expk
    = {\bf 0},\quad \forall \, {\bf k}, \,    \text{with }\,
    |{\bf k}|>\ind \}.
  \end{equation}
  The space ${\bf V}_\ind$ has finite dimension, denoted by
  $d_\ind$. Moreover, ${\bf V}_{\ind} \subset {\bf V}_{\ind+1}$ and,
  in the meaning of Hilbert spaces, $\mathbf{H}_{1} = \cup_{\ind \in
    \N^\star} {\bf V}_{\ind}$. We notice that ${\bf V}_\ind$ is a
  subset of the finite dimensional space
  \begin{equation*}
    {\bf W}_\ind := \{ \wit : \tore \rightarrow  \mathbb{C}^3, \,\, \wit = 
    \sum_{ {\bf k}\in {\cal T}_3, 
      | {\bf k} | \le \ind} \widehat {\bf w}_{\bf k}\, e^{+i\mathbf{k \cdot x}} \},
  \end{equation*}
  and the space $ {\bf V}_\ind$ can be described as
  \begin{equation}\label{TP0T6}
    {\bf V}_\ind := {\bf W}_\ind \cap {\bf H}_0.
  \end{equation}
  We denote by $({\bf e}_1, \dots, {\bf e}_{d_\ind})$ an orthogonal
  basis of ${\bf V}_\ind$. Remark that this basis is not made of the
  $e^{+i\mathbf{k \cdot x}}$'s. Nevertheless, we do not need to know
  it precisely. Moreover, the family $\{{\bf e}_j\}_{j \in \N}$ is an
  orthogonal basis of ${\bf H}_0$ as well as of ${\bf H}_1$. As we
  shall see in the following, the ${\bf e}_j$'s can be chosen to be
  eigen-vectors of $A$, with $\| {\bf e}_j \| =1$.  \par
  Let $\mathbb{P}_{\ind}$ denote the orthogonal projection from ${\bf
    H}_s$ ($s=0, 1$) onto ${\bf V}_\ind$. For instance, for $\bw_0=\overline
  {\vit_0} = \sum_{j=1}^{\infty} w_{j}^0 {\bf e}_j$, we have
  \begin{equation*}
    \mathbb{P}_{\ind} (\overline {\vit_0} ) = \sum_{j=1}^{d_\ind} w_{j}^0 {\bf e}_j.
  \end{equation*}
  In order to use classical tools for systems of ordinary differential
  equations, we approximate the external force by means of a standard
  Friederichs mollifier, see e.g.~\cite{LT75,Tem2001}. Let $\rho$ be an
  even function such that $\rho\in C^\infty_0(\R)$, $0\leq\rho(s)\leq
  1$, $\rho(s)=0$ for $|s|\geq1$, and $\int_{\R}\rho(s)\,ds=1$.  Then, set
  ${\bf F}(t)=\overline{\bff}(t)$ if $t\in[0,T]$ and zero elsewhere
  and for all positive $\epsilon$ define $\overline{\bff}_\epsilon$,
  the smooth (with respect to time) approximation of
  $\overline{\bff}$, by
  \begin{equation*}
    \overline{\bff}_\epsilon(t):=\frac{1}{\epsilon}\int_{\R}\rho\left(\frac{t-s}{\epsilon}
    \right)\,{\bf F}(s)\,ds.
  \end{equation*}
  Well known results imply that if~\eqref{GGMC0} is satisfied, then
  $\overline{\bff}_\epsilon\to\overline{\bff}$ in
  $L^2([0,T];\mathbf{H}_1)$.  Thanks to Cauchy-Lipschitz Theorem, we
  know the existence of a unique function
  \begin{equation*}
    \wit_\ind(t,{\bf x}) = \sum_{j=1}^{d_\ind} w_{\ind,j} (t)\,  {\bf e}_j({\bf x})
  \end{equation*}
  and of a positive $T_\ind$ such that the vector
  $(w_{\ind,1}(t),\dots,w_{\ind,d_\ind}(t))$ is a
  $C^1$ solution on $[0,T_\ind]\subseteq[0,T]$, with $w_{\ind,j} (0) =
  w_{\ind,j}^0$, in the sense that $\forall \, \vittest \in {\bf V}_\ind$, $\forall
  \, t \in [0,T_\ind]$ it holds
  \begin{equation}\label{RQPJ89}
    \begin{array} {l}
      \displaystyle \int_{\tore} \p_t \wit_\ind (t, \x) \cdot \vittest (\x)\, d\x  - 
      \int_{\tore} (\overline{D_N (\wit_\ind) \otimes D_N (\wit_\ind)})(t,
      \x) : \g  \vittest (\x)\, d\x \, 
      \\ 
      \displaystyle \hskip 4.5cm +\nu 
      \int_{\tore} \g \wit_\ind (t, \x) : \g  \vittest (\x)\, d\x \,  
      = \int_{\tore} \overline{\bf f}_{1/\ind} (t, \x) \cdot 
      \vittest (\x)\, d\x,
    \end{array} 
  \end{equation}
  where
  \begin{equation*}
    \p_t \wit_\ind = \sum_{j=1}^{d_\ind} \frac{d w_{\ind,j}(t)}{dt} \, {\bf e}_j .
  \end{equation*}
  As we shall see it in step 2, we can take $T_m = T$.  This ends the
  local-in-time construction of the approximate solutions
  $\bw_\ind(t,\x)$.  \fin
  \begin{remark}
    We want to stress to the reader's attention that a more precise
    notation would be
    \begin{equation*}
      \bw_{\ind,N,\alpha},
    \end{equation*}
    instead of $\bw_\ind$. We are asking for this simplification to
    avoid a too heavy notation, since in this section both $N$ and
    $\alpha$ are fixed.
  \end{remark}
  \bigskip

  \noindent\textcolor{blue}{\underline {\sc Step 2. Estimates}.}

  \medskip

  As in the classical Galerkin method, we need some \textit{a priori}
  estimates, first to show that the solution of the $(d_\ind\times
  d_\ind)$-systems of ordinary differential equations satisfied by
  $w_{\ind,j}$ exists in some non-vanishing time-interval, not
  depending on $\ind$ (to this end a energy-type estimate is
  enough). Next, we want to obtain estimates on the $\wit_\ind$'s and
  the $\p_t\wit_\ind$ for compactness properties, to pass to the
  limit, when $\ind \to \infty$ and $N$ is still kept fixed. \par
  As usual, we need to identify suitable test vector fields
  in~(\ref{RQPJ89}) such that, the scalar product with the nonlinear
  term vanishes (if such choice does exist). We observe that the
  natural candidate is $AD_N (\wit_\ind)$. Indeed, since $A$ is
  self-adjoint and commutes with differential operators, it holds:
  \begin{equation*}
    \begin{array} {l}
      \displaystyle  \int_{\tore} (\overline{D_N (\wit_\ind) 
        \otimes D_N (\wit_\ind)})
      : \g  (AD_N (\wit_\ind))
      \,d\x 
      \\
      = \displaystyle  \int_{\tore} G\big({D_N (\wit_\ind) 
        \otimes D_N (\wit_\ind)}\big)
      : \g  (AD_N (\wit_\ind) )
      \,d\x 
      \\ 
      = \displaystyle  \int_{\tore} (AG)\big({D_N (\wit_\ind) 
        \otimes D_N (\wit_\ind)}\big)
      : \g  (D_N (\wit_\ind) )
      \, d\x = 0,
    \end{array}
  \end{equation*}
  because $A\circ G = {\rm Id}$ on $\mathbf{H}_s$, $\g \cdot \big(D_N
  (\wit_\ind)\big)= 0$, and thanks to the periodicity.  This yields
  the equality
  \begin{equation}\label{TTFX34}
    \big(\p_t \wit_m , AD_N (\wit_\ind) \big) - \nu \big(\Delta \wit_\ind,
    AD_N (\wit_\ind)\big) = (\overline {\bf f}_{1/\ind} , AD_N (\wit_\ind)\big).
  \end{equation}
  This formal computation asks for two clarifications:
  \begin{enumerate}
  \item[i)] We must check that $AD_N (\wit_\ind)$ is a ``legal'' test
    function, to justify the above formal procedure. This means that
    for any fixed time $t$, we must prove that $AD_N (\wit_\ind) \in
    {\bf V}_\ind$.
  \item[ii)] Estimate~(\ref{TTFX34}) does not give a direct
    information about $\wit_\ind$ itself and/or $\p_t
    \wit_\ind$. Therefore one must find how to deduce suitable estimates
    from it.
  \end{enumerate}

  \medskip

  Point i) is the most simple to handle. On one hand, we already know
  that $G({\bf H}_0) = {\bf H}_2 \subset {\bf H}_0$. On the other
  hand, formula~(\ref{eq:123LKG}) makes sure that $G({\bf W}_\ind)
  \subset {\bf W}_\ind$. We now use representation~(\ref{TP0T6}), and
  we deduce that $G({\bf V}_\ind ) \subset {\bf V}_\ind$. Finally, it
  is clear that $\hbox{Ker}(G) = {\bf 0}$. Therefore, since ${\bf
    V}_\ind$ has a finite dimension, we deduce that $G$ is an
  isomorphism on it. Then the space ${\bf V}_\ind$ is stable under the
  action of the operator $A$ as well as under that of $D_N$. This
  makes $AD_N (\wit_\ind) (t, \cdot) \in {\bf V}_\ind$ a ``legal''
  multiplier in formulation~(\ref{RQPJ89}), for each fixed
  $t$. Moreover, since $A$ and $D_N$ are self-adjoint operator that
  commute, one can choose the basis $({\bf e}_1, \cdots, {\bf
    e}_{d_\ind}, \cdots )$ such that each ${\bf e}_j$ is still an
  eigen-vector of the operator $A$ and $D_N$ together. Therefore, the
  projection $\mathbb{P}_{\ind}$ commutes with $A$ as well as with all
  by-products of $A$, such as $D_N$ for instance. We shall use this
  remark later in the estimates.

  \medskip

  The next point ii) is not so direct and constitutes the heart of the
  matter of this paper. The key observation is that the following
  identities hold:
  \begin{eqnarray}
    && \label{TTY1} \big(\p_t \wit_\ind , AD_N (\wit_\ind)
    \big) = {1 \over 2}{d \over dt} \| A^{1 \over 2}D_N^{1 \over 2}(
    \wit_\ind)\|^2,
    \\
    && \label{TTY2} \big(-\Delta \wit_\ind, AD_N (\wit_\ind)\big) = \| \g
    A^{1 \over 2}D_N^{1 \over 2} (\wit_\ind) \|^2,
    \\
    && \label{TTY3} \big(\overline {\bf f}_{1/\ind} , AD_N (\wit_\ind)\big) =
    \big( A^{1 \over 2}D_N^{1 \over 2}( \overline {\bf f}_{1/\ind}), A^{1 \over
      2}D_N^{1 \over 2}( \wit_\ind) \big).
  \end{eqnarray}
  These equalities are straightforward because $A$ and $D_N$ both
  commute, as well as they do with all differential
  operators. Therefore~(\ref{TTFX34}) or better
  \begin{equation*}
    {1 \over 2}{d \over dt} \| A^{1 \over 2}D_N^{1 \over 2}(
    \wit_\ind)\|^2 
    +\| \g A^{1 \over 2}D_N^{1 \over 2}
    (\wit_\ind) \|^2= ( A^{1 \over 2}D_N^{1 \over 2}\big( \overline {\bf f}_{1/\ind}), 
    A^{1 \over 2}D_N^{1 \over 2}( \wit_\ind) \big)
  \end{equation*}
  shows that the natural quantity under control is $A^{1 \over
    2}D_N^{1 \over 2} (\wit_\ind)$. As we shall see in the remainder,
  norms of this quantity do control $\wit_\ind$, as well as the natural
  key variable $D_N(\wit_\ind)$. Finally, this yields an estimate for
  $\p_t \wit_\ind$.

  \medskip

  Since we need to prove many \textit{a priori} estimates, for the
  reader's convenience we organize the results in the following
  Table~\eqref{eq:Estimates}.  We hope that having a bunch of
  estimates collected together will help in understanding the
  result. In a first reading one can skip the proof of the
  inequalities, in order to get directly into the core of the main
  result. \par
  The results are organised as follows: In the first column we have
  labeled the estimates. The second column precises the variable under
  concern. The third one explains the bound in term of space
  function. The title of the space means that the considered sequence
  is bounded in this space. To be more precise, $E_{\ind,N} \in F$ for
  any variable $E_{\ind,N}$ and a space $F$, means that the sequence
  $\{E_{\ind,N}\}_{(\ind,N) \in \N^2}$ is bounded in the space $F$.
  Finally the fourth column precises the order in terms of $\alpha$,
  $\ind$, and $N$ for each bound. Of course each bound is of order of
  magnitude
  \begin{equation*}
    O\Big( \| \vit _0\|_{L^2} + { 1 \over \nu } \| {\bf f} \|_{L^2([0,T];L^2)}\Big),
  \end{equation*}
  and this why we do not mention it in the table. We also stress that
  all bounds are uniform in $\ind$. All bounds
  except~(\ref{eq:Estimates}-h) are uniform also in $N$. Moreover, we
  shall also see that they are uniform in $T$ yielding, $T_m = T$ for
  each $T$. We mention that we could take $T = \infty$ at this level
  of our analysis, so far the source term ${\bf f}$ is defined on $[0,
  \infty[$.
  \begin{equation}\label{eq:Estimates}
    \begin{array} {  @{}>{\vline \hfill  } c @{\,
          \, \,  }>{\vline  \hfill} c @{\, \, \,  }>{\vline \hfill } c
        @{\, \, \,  }>{\vline \, \, \,  } c   @{\hfill \vline  }} 
      \hline  
      \,\,  \hbox{Label} & \hbox{Variable} &\hbox{   bound          }
      \hskip 2cm ~  \hfill &      \hfill\hbox{order} \quad ~  
      \\ 
      \hline 
       \,a)\ \  \phantom{} & \, A^{1 \over 2}D_N^{1 \over 2} (\wit_\ind) & \, \,
      L^\infty([0,T];  \mathbf{H}_{0}) \cap L^2([0,T];  \mathbf{H}_{1})   &
      \hskip 1cm O(1)   
      \\ 
      \hline b)\ \  \phantom{}& D_N^{1/2}(\wit_\ind) \, \, \,  & L^\infty([0,T];
      \mathbf{H}_{0}) \cap L^2([0,T];  \mathbf{H}_{1})   & \hskip 1cm O(1)  
      \\
      \hline c )\ \  \phantom{}& D_N^{1/2}(\wit_\ind) \, \, \,  & L^\infty([0,T];
      \mathbf{H}_{1}) \cap L^2([0,T];  \mathbf{H}_{2})   & \hskip 1cm
      O(\alpha^{-1})
      \\ 
      \hline d)\ \  \phantom{} & \wit_\ind \, \, \, \, \, \,  & L^\infty([0,T];
      \mathbf{H}_{0}) \cap L^2([0,T];  \mathbf{H}_{1})   &\hskip 1cm O(1)
      \\
      \hline
      e )\ \  \phantom{}& \wit_\ind \, \, \, \, \, \,  & L^\infty([0,T];  \mathbf{H}_{1})
      \cap L^2([0,T];  \mathbf{H}_{2})   & \hskip 1cm O(\alpha^{-1}) 
      \\ 
      \hline 
      f)\ \  \phantom{} & D_N (\wit_\ind) \, \, \, & L^\infty([0,T];
      \mathbf{H}_{0}) \cap L^2([0,T];  \mathbf{H}_{1})   & \hskip 1cm O(1) 
      \\
      \hline 
      g)\ \  \phantom{} & D_N (\wit_\ind) \, \, \, & L^\infty([0,T];
      \mathbf{H}_{1}) \cap L^2([0,T];  \mathbf{H}_{2})   & O(\alpha^{-1} +
      (N+1)) \,~ 
      \\
      \hline 
      h)\ \  \phantom{} &\p_t \wit_\ind \, \, \, & L^2([0,T];
      \mathbf{H}_{0})   \hskip 1cm ~  &  \hskip 1cm O(\alpha^{-1}) 
      \\ 
      \hline 
    \end{array}
  \end{equation}

  \medskip

  \underline{\sl Checking~(\ref{eq:Estimates}-a) } --- For the
  simplicity, we shall assume that ${\bf f}\in L^2([0, T] \times
  \tore)^3$, but the proof holds true also for ${\bf f}\in L^2([0,
  T];\mathbf{H}_{-1})$: one has to substitute in~\eqref{TTY3} the
  integral over $\tore$ with the duality pairing $\langle\,.\,\rangle$
  between $\mathbf{H}_1$ and $\mathbf{H}_{-1}$ and estimate in a
  standard way the quantity $\big\langle\overline {\bf f}_{1/\ind} , AD_N
  (\wit_\ind)\big\rangle = \big\langle A^{1 \over 2}D_N^{1 \over 2}(
  \overline {\bf f}_{1/\ind}), A^{1 \over 2}D_N^{1 \over 2}( \wit_\ind)
  \big\rangle$. When one integrates~(\ref{TTFX34}) with respect to
  time on the time interval $[0,t]$ for any time $t\le T_m$, (by
  using~(\ref{TTY1})-(\ref{TTY2})-(\ref{TTY3}), and Cauchy-Schwartz
  inequality) one gets

  \begin{equation}\label{HHW12}
    \begin{aligned}
      \frac{1}{2}\|A^{1 \over 2} D_N^{1 \over 2} (\wit_\ind)(t, \cdot)
      \|^2 &+ \nu \int_0^t \|\g A^{1 \over 2}D_N^{1 \over 2}
      (\wit_\ind) \|^2\,d\tau
      \\
      &\le \frac{1}{2}\|A^{1 \over 2}D_N^{1 \over 2} \mathbb{P}_{\ind}
      \overline{\vit_0}\|^2 + \int_0^t\| A^{1 \over 2}D_N^{1 \over 2}
      (\overline {\bf f}_{1/\ind}) \| \cdot \| A^{1 \over 2}D_N^{1 \over 2}
      (\wit_\ind) \|\,d\tau.
    \end{aligned}
  \end{equation}

  Notice that $A^{1 \over 2}D_N^{1 \over 2}( \overline {\bf
    f}_\epsilon) = A^{-{1 \over 2}}D_N^{1 \over 2}( {\bf
    f}_\epsilon)$. Since the operator $ A^{-{1 \over 2}}D_N^{1 \over
    2}$ has for symbol $\rho_{N, {\bf k}}^{1/2} \le 1$, then $ \| A^{1
    \over 2}D_N^{1 \over 2} \overline {\bf f}_\epsilon \| \le C\|{\bf f} \|$
  (cf. also Remark~\ref{rem:fractional} and the properties of
  classical mollifiers). Since $\mathbb{P}_{\ind}$
  commutes with $A$ and $D_N$ we also have
  \begin{equation*}
    \|A^{1 \over 2}D_N^{1 \over 2}
    \mathbb{P}_{\ind} \overline{\vit_0}\| = \| \mathbb{P}_{\ind}A^{1 \over
      2}D_N^{1 \over 2} \overline{\vit_0}\| \le \| A^{1 \over 2}D_N^{1
      \over 2} \overline{\vit_0}\| \le \| \vit_0\|.
  \end{equation*}
  By using Poincar\'e's inequality combined with Young's inequality,
  and standard properties of mollifiers, one gets
  \begin{equation} \label{RROPL} 
    \frac{1}{2}\|A^{1 \over 2} D_N^{1 \over
      2} (\wit_\ind)(t, \cdot) \|^2 + {\nu \over 2} \int_0^t \|\g A^{1
      \over 2}D_N^{1 \over 2} (\wit_\ind) \|^2\,d\tau \le C ( \| \vit_0
    \|,\, \| {\bf f} \|_{L^2 ([0,T]; \mathbf{H}_{-1}) }).
  \end{equation}
  When one returns back to the definition of $\wit_\ind$, one obtain
  (as a by product of~(\ref{RROPL}) and also because the ${\bf e}_j$'s
  are eigen-vectors for $A$ and $D_N$ and therefore for $A^{1 \over 2}
  D_N^{1 \over 2}$) that
  \begin{equation*}
    \sum_{j= 1}^{d_m} \rho_{N, {j}} w_{m,j} (t)^2
    \le C ( \| \vit_0 \|, \| {\bf f} \|_{L^2 ([0,T]; \mathbf{H}_{-1}) }),
  \end{equation*}
  making sure that one can take $T_m = T$, since no $\rho_{N, {j}}$
  vanishes.

  \medskip

  \underline{\sl Checking~(\ref{eq:Estimates}-b)-(\ref
    {eq:Estimates}-c)} --- Let $\vittest \in {\bf H}_2 $. Then, with
  obvious notations one has
  \begin{equation*}
    \|A^{1 \over 2} \vittest \|^2 = \sum_{{\bf k} \in {\cal T}_3^\star }
    (1 + \alpha^2 | {\bf k} |^2) | \widehat \vittest_{\bf k}|^2 = \|
    \vittest \|^2 + \alpha^2 \|\g \vittest \|^2.
  \end{equation*}
  It suffices to apply this identity to $\vittest = D_N^{1 \over 2}
  (\wit_\ind)$ and to $\vittest = \p_i D_N^{1 \over 2} (\wit_\ind)$
  ($i=1,2,3$) in~(\ref{HHW12}) to get the claimed result.

  \medskip

  \underline{\sl
    Checking~(\ref{eq:Estimates}-d)-(\ref{eq:Estimates}-e)} --- This
  is a direct consequence
  of~(\ref{eq:Estimates}-b)-(\ref{eq:Estimates}-c) combined
  with~(\ref{IINVC9}), that we also can understand as
  \begin{equation*}
    \| \wit \|_s \le  \| D_N (\wit) \|_s \le 
    (N+1)  \| \wit \|_s,
  \end{equation*}
  for general $\wit$ and for any $s \ge 0$.  This explains how important
  is to have a ``lower bound'' for the operator $D_N$, since estimates
  on $D_N^{\alpha}\bw$, $\alpha>0$, are in inherited by $\bw$.

  \medskip

  \underline{\sl Checking~(\ref{eq:Estimates}-f)} --- The operator
  $A^{{1 \over 2}}D_N^{1 \over 2}$ has for symbol $(1+ \alpha^2 | {\bf
    k} |^2)\rho_{N, {\bf k}}^{1/2} $ while the one of $D_N$ is $(1+
  \alpha^2 | {\bf k} |^2)\rho_{N, {\bf k}}$. Since $ 0 \le \rho_{N,
    {\bf k} } \le 1$, then $\| D_N (\wit) \|_s \le \| A^{{1 \over
      2}}D_N^{1 \over 2}(\bw) \|_s$ for general $\wit$ and for any $s
  \ge 0$.  Therefore, the estimate~(\ref{eq:Estimates}-f) is still a
  consequence of~(\ref{eq:Estimates}-a).  \medskip

  \underline{\sl Checking~(\ref{eq:Estimates}-g)} --- This follows
  directly from~(\ref{eq:Estimates}-e) together
  with~(\ref{IINVC9}). This also explains why the result depends on
  $N$, since we are using now the \textit{upper bound} on the norm of
  the operator $D_N$, while in the previous estimates, we used
  directly the equation as well as the \textit{lower bound} for the
  Van Cittert operator $\bn$.
  \begin{remark}
    The fact that this estimate is valid for each $N$, but the bound
    may grow with $N$ is the main source of difficulties in passing to
    the limit as $N\to+\infty$. Also the lack of this uniform bound
    requires some work to show that (certain sequences of) solutions
    to~\eqref{eq:adm1} converge to the average of a weak (dissipative)
    solution of the Navier-Stokes equations.
  \end{remark}

  \medskip

  \underline{\sl Checking~(\ref{eq:Estimates}-h)} --- Let us take
  $\p_t \wit_\ind \in {\bf V}_\ind$ as test vector field
  in~(\ref{RQPJ89}). We get
  \begin{equation*}
    \|\p_t \wit_\ind \|^2 + \int_{\tore} {\bf A}_{N,\ind}
    \cdot \p_t \wit_\ind + \frac{\nu}{2}{d \over dt} \| \g \wit_\ind \|^2 =
    \int_{\tore} \overline{\bf f}_{1/\ind} \cdot \p_t \wit_\ind,
  \end{equation*}
  where
  \begin{equation} \label{RTP09} {\bf A}_{N,\ind} := \overline {\g
      \cdot \big(D_N (\wit_\ind) \otimes D_N( \wit_\ind)\big)} .
  \end{equation}
  So far $\wit _\ind (0, \cdot) = \mathbb{P}_{\ind} (\overline {
    \vit_0})\in {\bf H}_2$ and obviously $\| \mathbb{P}_{\ind}
  (\overline { \vit_0}) \|_2 \le C \alpha^{-1}\| \vit_0\|$, we only
  have to check that ${\bf A}_{N,\ind} $ is bounded in $L^2
  ([0,T]\times \tore)^3$.  Thanks to~(\ref{eq:Estimates}-f), it is
  easy checked with classical interpolation inequalities, that $D_N
  (\wit_\ind) \in L^4 ([0,T];L^3(\tore)^3)$.  Therefore, $ D_N
  (\wit_\ind) \otimes D_N (\wit_\ind) \in L^2 ([0,T];
  L^{3/2}(\tore)^9)$.  Because the operator $(\g \cdot) \circ G$ makes
  to ``gain one derivative,'' we deduce that ${\bf A}_{N,\ind} \in L^2
  ([0,T]; W^{1, 3/2}(\tore)^3)$, which yields to ${\bf A}_{N,\ind} \in
  L^2 ([0,T]\times \tore)^3$ since $W^{1, 3/2}(\tore) \subset
  L^3(\tore) \subset L^2(\tore)$ and $L^2([0,T]; L^2 (\tore)^3)$ is
  isomorphic to $L^2 ([0,T]\times \tore)^3$
  (see~\cite{RLbook}). Moreover, the bound is of order
  $O(\alpha^{-1})$ as well, because of the norm of the operator $(\g
  \cdot) \circ G$ that we do not need to specify, but where $\alpha$
  is involved.  The remainder of the proof is a very classical trick,
  based on Gronwall's lemma (see for instance in~\cite{RL09} for a
  detailed report about this method). Notice that this bound on the
  nonlinear term is not optimal, but fits with our requirements.  \fin

  \bigskip

  \noindent\textcolor{blue}{\underline{\sc Step 3 : Passing to the
      limit in the equations when $\ind \to \infty$, and $N$ is
      fixed}.}

  \medskip

  Thanks to the bounds~(\ref{eq:Estimates}) and classical tricks, we
  can extract from the sequence $\{\wit_\ind\}_{n \in \N}$ a
  sub-sequence converging to a $\wit \in L^\infty([0,T];
  \mathbf{H}_{1}) \cap L^2([0,T]; \mathbf{H}_{2})$. Using Aubin-Lions
  Lemma (here one uses~(\ref{eq:Estimates}-h) and again classical
  tricks), one has:
  \begin{eqnarray}
    && \wit_\ind \rightarrow \wit  \hbox{ weakly in }
    L^2([0,T];  \mathbf{H}_{2}), 
    \\
    && \label{0HE3} \wit_\ind \rightarrow \wit \hbox{ strongly in }
    L^p([0,T];  \mathbf{H}_{1}) ,\quad \forall\,  p<\infty, 
    \\
    && \p_t \wit_\ind \rightarrow \p_t \wit \hbox{ weakly in } L^2([0,T];
    \mathbf{H}_{0}).  
  \end{eqnarray} 
  This already implies that $\wit$
  satisfies~(\ref{Reg11})-(\ref{Reg12}). From~(\ref{0HE3}) and the
  continuity of $D_N$ in ${\bf H}_s$ we get that $D_N (\wit_\ind)$
  converges strongly to $D_N (\wit)$ in $L^4 ([0,T] \times
  \tore)$. Then, $D_N(\wit_\ind) \otimes D_N (\wit_\ind) $ converges
  strongly to $D_N(\wit) \otimes D_N (\wit) $ in $L^2 ([0,T] \times
  \tore)$. This convergence of $\bw$, together with the fact that
  $\overline{\bff}_{1/\ind}$ converges strongly, implies that for all $\vittest \in
  L^2([0,T] ; {\bf H}_1)$
  \begin{equation} \label{VARFORM}
    \begin{array} {l}
      \displaystyle\int_0^T\int_{\tore} \p_t \wit
      \cdot \vittest \,d\x\, d\tau- \int_0^T\int_{\tore} \overline{D_N (\wit) \otimes
        D_N (\wit)} : \g
      \vittest \,d\x\, d\tau\\
      \hskip 4.5cm+ \displaystyle \nu \int_0^T\int_{\tore} \g \wit : \g \vittest \,d\x\, d\tau 
           =\int_0^T\int_{\tore} \overline{\bf f}
      \cdot \vittest \,d\x\, d\tau. 
    \end{array}
  \end{equation}
  Arguing similarly to~\cite{RL09}, we easily get that $\wit$
  satisfies~(\ref{Initial}).

  We are almost in order to introduce the pressure.  Before doing
  this, let us make so comments about the variational formulation
  above.  We decided to take test vector fields in $L^2 ([0,T];
  \mathbf{H}_1)$ to be in accordance with classical presentations. The
  regularity of $\wit$ however yields $\overline { \g \cdot \big(D_N
    (\wit) \otimes D_N (\wit)\big) } \in L^2 ([0,T]\times \tore)^3 $
  as well as $\Delta \wit \in L^2 ([0,T]\times \tore)^3 $.
  Consequently, one can take vector test fields $\vittest \in
  L^2([0,T] ; {\bf H}_0)$ in formulation~(\ref{VARFORM}) that we can
  rephrased as: $\forall \, \vittest \in L^2([0,T] ; {\bf H}_0)$,
  \begin{equation} \label{VARFORM2} 
    \displaystyle\int_0^T\int_{\tore}
    (\p_t \wit + {\bf A}_N - \nu \Delta \wit - \overline {\bf f} )
    \cdot \vittest \, d\x\, d\tau = 0,
  \end{equation} 
  where for convenience, we have set
  \begin{equation*}
    {\bf A}_N := \overline { \g \cdot \big(D_N
      (\wit) \otimes D_N (\wit)\big) }.
  \end{equation*} 
  Therefore, for almost every $t \in [0,T]$, 
  \begin{equation*}
    \mathbb{F}(t, \cdot)
    = (\p_t \wit + {\bf A}_N - \nu \Delta \vit - \overline {\bf f}) (t,
    \cdot) \in L^2(\tore)^3
  \end{equation*}
  is orthogonal to divergence-free vector fields in $L^2(\tore)^3$ and
  De Rham's Theorem applies.  Before going into technical details, let
  us first recall that among all available versions of this theorem,
  the most understandable is the one given by L.~Tartar
  in~\cite{LT75}, a work that has been later reproduced in many other
  papers. Notice also that there is a very elementary proof in the
  periodic case~\cite{RLbook}. Now, from~(\ref{VARFORM2}), we deduce
  that for each Lebesgue point $t$ of $\mathbb{F}$, there is a scalar
  function $q(t, \cdot) \in H^1(\tore)$, such that $\mathbb{F} = - \g
  q$, that one can rephrase as
  \begin{equation}\label{EQ:Final}
    \p_t \wit + {\bf A}_N - \nu \Delta \wit + \g q =
    \overline {\bf f}.
  \end{equation}
  Without loss of generality, one can assume that $\g \cdot {\bf f} =
  0$. Therefore, taking the divergence of equation~(\ref{EQ:Final})
  yields
  \begin{equation*}
    \Delta q = \g \cdot {\bf A}_N,
  \end{equation*}
  which also easily yields $q \in L^2 ([0,T]; H^1(\tore))$,
  closing this part of the construction.  \fin

  \bigskip

  \noindent \textcolor{blue} {\underline{\sc Step 4} : About the
    initial data}.  \medskip

  We must check that $\wit(0, \cdot)$ can be defined and that $\wit(0,
  \cdot) = \overline{\vit_0}$, such as defined
  in~(\ref{Initial}). Thanks to the estimates above (mainly $\p_t \wit
  \in L^2([0,T]; \mathbf{H}_{0})$, together with the regularity about
  $\wit$ that we get) one obviously has $\wit \in C([0,T];
  \mathbf{H}_{1})$.  This allows us to define $\wit(0, \cdot) \in
  \mathbf{H}_{1}$ and also to guarantee
  \begin{equation*}
    \displaystyle \lim_{t \rightarrow 0^+}\|\wit(t, \cdot) -  \wit(0, \cdot)   \|_{ \mathbf{H}_{1}} = 0.
  \end{equation*} 
  It remains to identify $ \wit(0, \cdot)$.  The construction
  displayed in Step 1 yields for $m \in \N$, 
  \begin{equation}\label{INITCOND}
  \wit_\ind (t, \x) = \mathbb{P}_{\ind} (\overline {\vit_0} )(\x) +
  \int_0^t \p_t \wit_\ind (s, \x)\,ds,
  \end{equation}
  an identity which holds in $C^1([0,T] \times \Om)$. Because of the
  weak convergence of $(\p_t \wit_\ind)_{\ind \in \N}$ to $\p_t \wit$
  in $L^2 ([0,T]; \mathbf{H}_{0})$ and thanks to usual properties of
  the projection's operator $\mathbb{P}_{\ind}$, one can easily pass
  to the limit in~(\ref{INITCOND}) in a weak sense in the space $L^2
  ([0,T]; \mathbf{H}_{0})$, to obtain
  \begin{equation*}
    \wit (t, \x) =    \overline {\vit_0} (\x) + \int_0^t \p_t \wit (s, \x)\,ds.
  \end{equation*}
  Therefore, $\wit(0, \x) = \overline {\vit_0} (\x)$, closing the
  question about the initial data.  \fin

 \bigskip

  \noindent\textcolor{blue}{\underline{{\sc Step 5}}: Uniqueness}.

  \medskip

  Let $\wit_1$ and $\wit_2$ be two solutions corresponding to the same
  data $(\bu_0,\,\bff)$ and let us define as usual ${\bf W} := \wit_1-
  \wit_2$. We want to take $AD_N ({\bf W})$ as test function in the
  equation satisfied by ${\bf W}$, because we suspect that this is the
  natural multiplier for an energy equality. But before doing this, we
  must first check that this guy lives in $L^2 ([0,T]\times \tore)^3$
  to be sure that this is a ``legal'' multiplier.  Notice that $AD_N$
  has for symbol $(1+\alpha^2 |{\bf k}|^2) ^2 \rho_{N, {\bf k}}
  \approx (N+1) (1+\alpha^2 |{\bf k}|^2) ^2 / \alpha^2 |{\bf k}|^2
  \approx (N+1) \alpha^2 |{\bf k}|^2 $ for large $|{\bf
    k}|$. Therefore, for each fixed $N\in \N$, $AD_N$ is like a
  Laplacian and makes us loose two derivatives in space. (Only in the
  limit $N\to+\infty$ we will loose four derivatives!) Fortunately,
  ${\bf W}\in L^2([0,T]; {\bf H}_2)$ and therefore $AD_N( {\bf W}) \in
  L^2([0,T]\times \tore)^3$, making this guy a good candidate to be
  the multiplier we need. After using tricks already introduced in
  this paper, we get
  \begin{equation}\label{56LKP0}
    \begin{array}{l} 
      \displaystyle \frac{1}{2}{d \over  dt} \| A^{1 \over 2} D_N ^{1 \over 2}( {\bf W}) \|^2 + \nu
      \| \g A^{1 \over 2} D_N ^{1 \over 2}( {\bf W} )\|^2 \\
      \hskip 4.5cm  \begin{array} {lcl}  & \le
        & | ( (D_N ({\bf W}) \cdot \g ) D_N (\wit_2), D_N ({\bf W}) ) |, 
        \\
        & \le & \| D_N ({\bf W}) \|^2_{L^4(\tore)} \| \g D_N (\wit_2) \|,  
        \\
        & \le & \| D_N ({\bf W}) \|^{1/2} \| \g D_N ({\bf W}) \|^{3/2} \| \g D_N
        ( \wit_2) \|,
      \end{array}
    \end{array} 
  \end{equation}
  where the last line is obtained thanks the well-known
  ``Lady\v{z}henskaya inequality'' for interpolation of $L^4$ with
  $L^2$ and $H^1$, see~\cite[Ch.~1]{La69}. Starting from the last line
  of~(\ref{56LKP0}), we use the inequality $\| D_N( {\bf W}) \|\le
  \|A^{1 \over 2} D_N ^{1 \over 2}( {\bf W})\|$ together with $\| D_N
  (\g {\bf W}) \|\le \|A^{1 \over 2} D_N ^{1 \over 2} (\g {\bf W})\|$,
  the fact that $D_N$ and $\g$ commute, $\|D_N\| = (N+1)$, the bound
  of $\wit_2$ in $L^\infty([0,T]; {\bf H}_1)$, and Young's
  inequality. We obtain
  \begin{equation*}
    \begin{array}{l} 
      \displaystyle \frac{1}{2}{d \over  dt} \| A^{1 \over 2}
      D_N ^{1 \over 2}( {\bf W}) \|^2 + \nu \| \g A^{1 \over 2} D_N ^{1
        \over 2} ({\bf W}) \|^2 
      \\ 
      \displaystyle \hskip 4cm \leq{ 27 (N+1)^4 \sup_{t \ge 0} \| \g \wit_2
        \| \over 32 \nu^3} \| A^{1 \over 2}  D_N ^{1 \over 2} ({\bf W})
      \|^2 + {\nu \over 2} \| \g A^{1 \over 2}
      D_N ^{1 \over 2}( {\bf W}) \|^2.
    \end{array}
  \end{equation*}
  In particular, we get
  \begin{equation*}
    \frac{1}{2}{d \over  dt} \| A^{1 \over 2} D_N ^{1 \over 2}( {\bf W}) \|^2 \le 
    {  27 (N+1) ^4 \sup_{t \ge 0} \| \g \wit_2 \| \over 32 \nu^3} \  \|
    A^{1 \over 2} D_N ^{1 \over 2} ({\bf W}) \|^2. 
  \end{equation*}
  We deduce from Gronwall's Lemma that $A^{1 \over 2} D_N ^{1 \over
    2}({\bf W}) = {\bf 0}$ because $A^{1 \over 2} D_N ^{1 \over
    2}({\bf W}) (0, \cdot) = {\bf 0}$. To conclude that ${\bf W} =
  {\bf 0}$, we must show that the kernel of the operator $A^{1 \over
    2} D_N ^{1 \over 2}$ is reduced to ${\bf 0}$. This is trivial,
  since this operator has for symbol $(1+ \alpha^2 | {\bf k} |^2 )
  \rho_{N, {\bf k} } \approx \alpha | {\bf k} | $ for large values of
  ${\bf k}$. This symbol never vanishes and the equivalence at
  infinity shows that this operator is of same order of $\alpha | \g
  |$. Therefore, it is an isomorphism that maps ${\bf H}_s$ onto ${\bf
    H}_{s-1}$. This concludes that the considered kernel is reduced to
  zero, proving uniqueness of the solution.
\end{proof}
\begin{remark}\label{En:Approx} 
 As we have seen, $AD_N (\wit)$ is a legitimate test function in
 Equation~(\ref{EQ:Final}). When using computation rules already
 detailed in the paper, we get the following energy equality
 satisfied by $A^{1/2} D_N^{1/2} (\wit)$, 
 \begin{equation*}
  \frac{1}{2} {d \over  dt}\| A^{1/2} D_N^{1/2} (\wit) \|^2 + \nu \| \g A^{1/2}
   D_N^{1/2}( \wit) \|^2 = \big(A^{-1/2} D_N^{1/2} ({\bf f}), A^{1/2}
   D_N^{1/2}( \wit)\big).
\end{equation*}
As we shall see in the sequel it seems that it is not possible to pass
to the limit $N\to+\infty$ directly in this ``energy equality'' and
some work to obtain an ``energy inequality'' is needed.
\end{remark}
\section{Passing to the limit when $N \to \infty$}\label{sec:limit}
The aim of this section is the proof of main result of the paper,
namely Theorem~\ref{thm:Principal}.  We now denote, for a given
$N\in\N$ by $(\wit_N, q_N)$, the ``regular weak'' solution to
Problem~\ref{eq:adm1}.  For the sake of completeness and to avoid
possible confusion between the Galerkin index $\ind$ and the
deconvolution index $N$, we write again the system satisfied by
$\bw_N$
\begin{equation} \label{eq:adm-N}
 \begin{array}{rr}
   \partial_t \wit_N +\nabla\cdot(\overline {D_N(\wit_N) \otimes
     D_N(\wit_N)})-\nu\Delta\wit_N+\nabla q_N= \overline \bff & \hbox{in}
   \quad  [0,T]\times \tore, 
   \\
   \nabla\cdot\wit_N =0 & \hbox{in} \quad  [0,T]\times \tore , 
   \\ 
   \wit_N (0, \x) = \overline{\vit_0} (\x)&\hbox{in} \quad
   \tore. \hspace{1.15cm}  \phantom{1} 
 \end{array}
\end{equation}
(More precisely, for all $N\in\N$ we set
$\bw_N=\lim_{\ind\to+\infty}\bw_{\ind,N,\alpha})$, and the scale
$\alpha>0$ is fixed.) We aim to prove that the sequence $\{(\wit_N,
q_N)\}_{N \in \N}$ has a sub-sequence which converges to some $(\wit,
q)$ that is a solution to the averaged Navier-Stokes
Equations~(\ref{eq:adm2}).  Recall that $(A \wit, A q)$ will be a
distributional solution to the Navier-Stokes Equations.  This result
gives a undeniable theoretical support for the study of this ADM
model.

\medskip

We divide the proof into two steps:
\begin{enumerate} 
\item We search additional estimates, uniform in $N$, to get
 compactness properties about the sequences $\{D_N (\wit_N)\}_{N\in\N}$ and
 $\{\wit_N\}_{N\in\N}$;
\item We prove strong enough convergence in order to pass to the limit
  in the equation~\eqref{eq:adm-N}.
\end{enumerate} 
Of course the challenge in this process is to pass to the limit in the
nonlinear term $D_N (\wit_N) \otimes D_N (\wit_N)$. This is why we
seek for an estimate about $\p_t D_N (\wit_N)$, knowing that we
already have some estimate for $D_N (\wit_N)$. This is how we get a
compactness property satisfied by $D_N (\wit_N)$, that we use for
passing to the limit in the nonlinear term.

\bigskip

\noindent \textcolor{blue}{\underline{\sc Step 1 : Additional
    estimates}.}

\medskip

We quote in the following table the estimates we shall use for passing
to the limit. The Table~(\ref{Estimates22}) is organized as the
previous one~\eqref{eq:Estimates}. Recall that -for simplicity- $E_N \in
F$ for any variable $E_N$ and a space $F$, means that the sequence
$\{E_N\}_{N \in \N}$ is bounded in the space $F$.
\begin{equation}\label{Estimates22}  
 \begin{array}
{  @{}>{\vline \hfill  } c @{\, \, \,  }>{\vline
       \hfill} c @{\, \, \,  }>{\vline \hfill } c @{\, \, \,
     }>{\vline \, \, \,  } c   @{\hfill \vline  }} 
   \hline  
   \,\,  \hbox{Label} & \hbox{Variable} &\hbox{   bound          } \hskip 2cm ~  \hfill & 
   \hfill\hbox{order} \quad ~  
   \\ 
   \hline a & \wit_N \, \, \, \, \, \,  & L^\infty([0,T];
   \mathbf{H}_{0}) \cap L^2([0,T];  \mathbf{H}_{1}) \hskip 0,5 cm ~
   &\hskip 0,2cm O(1) 
   \\
   \hline
   b & \wit_N \, \, \, \, \, \,  & L^\infty([0,T];  \mathbf{H}_{1})
   \cap L^2([0,T];  \mathbf{H}_{2})   \hskip 0,5 cm ~& \hskip 0.2cm
   O(\alpha^{-1}) 
   \\
   \hline \, \,   c & D_N (\wit_N) \, \, \, & L^\infty([0,T];
   \mathbf{H}_{0}) \cap L^2([0,T];  \mathbf{H}_{1}) \hskip 0,5 cm ~ &
   \hskip 0,2cm O(1) 
   \\
   \hline \, \,   d &\p_t \wit_N \, \, \, &\, \,  L^2 ([0,T] \times
   \tore)^3   \hskip 1,2cm ~  & \hskip 0.2cm O(\alpha^{-1} ) \hskip
   0.5 cm  
   \\
   \hline \, \,   e & q_N \, \, \, \, \, \,  &\hskip 0,5cm
   L^{2}([0,T];  H^1(\tore)) \cap L^{5/3}([0,T];  W^{2, 5/3}(\tore)) \,   
   &\hskip 0,2cm O(\alpha^{-1} ) \hskip 0.2 cm ~
   \\
   \hline \, \,   f &\, \, \, \p_t D_N(\wit_N) \, \, \, &\, \,
   L^{4/3}([0,T];  {\bf H}_{-1})   \hskip 1,2cm ~  & \hskip 0.2cm O(1 )
   \hskip 0.5 cm 
   \\
   \hline
 \end{array}
\end{equation}
Estimates~(\ref{Estimates22}-a),~(\ref{Estimates22}-b),~(\ref{Estimates22}-c),
and~(\ref{Estimates22}-d) have already been obtained in the previous
section. Therefore, we just have to check~(\ref{Estimates22}-e)
and~(\ref{Estimates22}-f).

\medskip

\underline{\sl Checking~(\ref{Estimates22}-e) } --- As usual, to
obtain regularity properties of the pressure (at least in the
space-periodic case), we take the divergence of~(\ref{EQ:Final})
obtaining
\begin{equation*}
 -\Delta q_N = \g \cdot {\bf A}_N - \g \cdot \overline {\bf f},
\end{equation*}
where we recall that 
\begin{equation*}
 {\bf A}_N = \overline { \g \cdot (D_N (\wit_N) \otimes D_N (\wit_N)) }.
\end{equation*}
Next, since ${\bf f} \in L^2 ([0,T]\times \tore)^3$, then we get $\g
\cdot \overline {\bf f} \in L^2 ([0,T]; H^1(\tore)^3)$. %
We now investigate the regularity of ${\bf A}_N$. We already know from
the estimates proved in the previous section that ${\bf A}_N \in L^2
([0,T] \times \tore)^3$. This yields the first bound in $L^2 ([0,T];
H^1(\tore))$ for $q_N$.

\medskip

We now seek the other estimate for $q_N$. Classical interpolation
inequalities combined with~(\ref{Estimates22}-c) yield $D_N
(\wit_N) \in L^{10/3} ([0,T] \times \tore)$. Therefore, $ {\bf A}_N
\in L^{5/3}([0,T]; W^{1, 5/3}(\tore)) $. Consequently, we obtain
\begin{equation*}
  q_N \in L^2 ([0,T]; H^1(\tore))\cap  L^{5/3}([0,T];  W^{2, 5/3}(\tore)).
\end{equation*}

\medskip

\underline{\sl Checking~(\ref{Estimates22}-f) } --- Let be given
$\vittest \in L^4 ([0,T]; {\bf H}_1)$. We use $D_N (\vittest) \in L^4
([0,T]; {\bf H}_1)$ as test function in the equation satisfied by
$(\wit_N, q_N)$, that is now (by the results previously proved) a
completely justified computation. We get, by using that $\p_t \wit \in
L^2 ([0,T] \times \tore)^3$ (as well as all other guys in the
equation), $D_N$ commutes with differential operators, $G$ and $D_N$
are self-adjoint, the pressure term cancels because $\g \cdot D_N
(\vittest) = 0$, and classical integrations by parts
\begin{equation}\label{RRT4} 
 \begin{aligned}
   (\p_t \wit_N,&\, D_N (\vittest)) = 
   (\p_t D_N(\wit_N), \vittest)  
   \\
   &=\nu (\Delta \wit_N, D_N (\vittest)) - 
   ( D_N (\wit_N) \otimes D_N (\wit_N), \overline {D_N (\g \vittest)}) - 
   (D_N (\overline{ \bf f}), \vittest).  
 \end{aligned}  
\end{equation}
We first observe that 
\begin{equation}\label{FF1} 
 |(\Delta \wit_N, D_N (\vittest))| = |( \g D_N (\wit_N) ,\g \vittest ) | \le 
 C_1 (t) \| \vittest \|_1,
\end{equation}
and we use the $L^2([0,T]; H^1(\tore)^3)$ bound for $D_N(\wit_N)$, to
infer that the function $C_1(t) \in L^2([0,T])$, with a bound uniform
with respect to $N\in\N$. Using $\| D_N (\overline{\bf f}) \| \le \|
{\bf f} \|$ already proved in the previous section and Poincar\'e's
inequality, we handle the term concerning the external forcing as
follows:
\begin{equation}\label{FF2} 
 \big|(D_N (\overline{ \bf f}), \vittest)\big| \le C \,  \| {\bf f} \|
 \, \| \vittest \|_1, 
\end{equation}
$C$ being the Poincar\'e's constant. Finally,
from~(\ref{Estimates22}-c) and usual interpolation inequalities, we
obtain that $D_N (\wit_N)$ belongs to $L^{8/3} ([0,T]; L^4
(\tore)^3)$, which yields
\begin{equation*}
 D_N (\wit_N) \otimes D_N (\wit_N) \in L^{4/3} ([0,T]; L^2 (\tore)^9).
\end{equation*}
Therefore, when we combine the latter estimate with $\|\overline{D_N
 (\g \vittest)}\| \le \| \g \vittest\|$, to get
\begin{equation}\label{FF3} 
  \big|\big( D_N (\wit_N) \otimes D_N (\wit_N), \overline D_N (\g \vittest)\big)\big| 
  \le C_2 (t) \| \vittest \|_1,
\end{equation}
where $C_2 (t) \in L^{4/3} ([0,T])$ and it is uniform in $N\in\N$. The
final result is a consequence of~\eqref{RRT4} combined
with~\eqref{FF1},~\eqref{FF2},~\eqref{FF3}, and $C_1(t) + \| {\bf
 f}(t, \cdot) \| + C_2 (t)\in L^{4/3} ([0,T])$, uniformly with respect to
$N\in\N$.\ \fin

\medskip 

\textcolor{blue}{\underline{\sc Step 2 : Passing to the limit}.}

\medskip

From the above estimates  and classical rules of functional analysis,
we can infer that there exist 
\begin{equation*}
 \begin{aligned}
   &\wit \in L^\infty([0,T];  \mathbf{H}_{1}) \cap L^2([0,T];\mathbf{H}_{2}),
   \\  
   &{\bf z}\in L^\infty([0,T];  \mathbf{H}_{0}) \cap L^2([0,T];\mathbf{H}_{1}),
   \\
   &q \in L^2 ([0,T];  H^1(\tore))\cap L^{5/3}([0,T];W^{2,5/3}(\tore)),
 \end{aligned}
\end{equation*}
such that, up to sub-sequences,
\begin{equation}\label{conv:res}
  \begin{array}{ll} 
    \wit_N \longrightarrow \wit  & 
    \left \{
      \begin{array} {l} 
        \hbox{weakly in }  L^2([0,T];  \mathbf{H}_{2}), 
        \\
        \hbox{weakly$\ast$ in }  L^\infty([0,T];  \mathbf{H}_{1}), 
        \\
        \hbox{strongly in } L^p([0,T]; {\bf H}_1), \quad \forall \, p < \infty, 
      \end{array}  
    \right. 
    \\
    \\
    \p_t \wit_N \longrightarrow \p_t \wit & \hbox{weakly in }
    L^2([0,T] \times \tore),  
    \\
    \\
    D_N(\wit_N) \longrightarrow {\bf z}  & 
    \left \{ 
      \begin{array} {l} 
        \hbox{weakly in }  L^2([0,T];  \mathbf{H}_{1}), 
        \\
        \hbox{weakly$\ast$ in }  L^\infty([0,T];  \mathbf{H}_{0}),          
        \\
        \hbox{strongly in } L^p([0,T] \times \tore)^3, \  \forall \, p < 10/3, 
      \end{array}  
    \right. 
    \\
    \\
    \p_t D_N (\wit_N) \longrightarrow \p_t {\bf z} & \hbox{weakly in }
    L^{4/3}([0,T];\mathbf{H}_{-1}),   
    \\
    \\
    q_N \longrightarrow q & \hbox{weakly in } L^2([0,T];
    H^1(\tore))\cap L^{5/3}([0,T];  W^{2, 5/3}(\tore)).
  \end{array}
\end{equation}
Of course, we have 
\begin{equation}\label{NonLin:8}
 D_N\wit_N \otimes D_N\wit_N \longrightarrow {\bf z} \otimes {\bf z} 
 \quad \hbox{strongly in } L^p([0,T] \times \tore)^9, \quad \forall \,
 p < 5/3. 
\end{equation}
All other terms in the equation pass easily to the limit as well. Our
proof will be complete as soon as we shall have checked that ${\bf z}
= A \wit$, thanks to~(\ref{NonLin:8}). However, this is almost
straightforward as we shall see, the hard job being already done by
the proof of the various estimates.  \par
Let us consider $\vittest \in L^2([0,T]; {\bf H}_2)$. We have $(D_N
(\wit_N) , \vittest) = (\wit_N, D_N (\vittest))$. We claim that $D_N
(\vittest) \to A \vittest$ strongly in $L^2([0,T] \times \tore)^3$. This
will suffice to conclude the proof. Indeed, if such a convergence
result holds, we have using~(\ref{conv:res}) and computational rules
already quoted,
\begin{equation*}
 \begin{CD}
   \big(D_N (\wit_N) ,\vittest\big) @=\big(\wit_N, D_N (\vittest)\big)
   \\
   @VVV @VVV\\
   ({\bf z},\vittest)@. (\bw,A\vittest)
   \\
   @| @|
   \\
   ({\bf z},\vittest)@= (A\bw,\vittest)
 \end{CD}
\end{equation*}
yielding ${\bf z} = A \wit$.\par
It remains to check the claim. We can write
\begin{equation*}
 \vittest = \sum_{\mathbf{k} \in {\cal T}_3^\star}\widehat{
   \vittest}_{\bf k}(t)\,e^{+i\mathbf{k \cdot x}}, 
\end{equation*}
and consequently (see in~\cite{RLbook})
\begin{equation*}
 \| \vittest \|_{L^2([0,T]; {\bf H}_2)}^2 = \sum_{\mathbf{k} \in {\cal T}_3^\star} |{\bf k} |^4
 \int_0^T |\widehat{ \vittest}_{\bf k}(t)|^2dt < \infty.
\end{equation*}
Let $\E >0$ being given. Then, there exists $0<K=K(\vittest)\in\N$ such that
\begin{equation*}
 \sum_{|\mathbf{k}| > K }
 |{\bf k} |^4 \int_0^T |\widehat{ \vittest}_{\bf k}(t)|^2 dt < \frac{\E}{2}.
\end{equation*}
Since $0\leq(1-\rho_{N,\bk})\leq1$, we have 
\begin{equation*}
 \begin{aligned}
   & \int_0^T \| (A-D_N) \vittest \|^2 = \sum_{\mathbf{k} \in {\cal
       T}_3^\star} (1 + \alpha^2 |{\bf k} |^2 )^2 (1- \rho_{N, {\bf
       k}})^2 \int_0^T |\widehat{ \vittest}_{\bf k}(t)|^2 dt,
   \\
   &\qquad= \sum_{0<|\mathbf{k}| \le K } (1 + \alpha^2 |{\bf k} |^2
   )^2 (1- \rho_{N, {\bf k}})^2 \int_0^T |\widehat{ \vittest}_{\bf
     k}(t)|^2 dt
   \\
   &\qquad\quad+ \sum_{|\mathbf{k}| > K }(1 + \alpha^2 |{\bf k} |^2
   )^2 (1- \rho_{N, {\bf k}})^2 \int_0^T |\widehat{ \vittest}_{\bf
     k}(t)|^2 dt, 
   \\& 
   \qquad< \sum_{0<|\mathbf{k}| \le K } (1 +
   \alpha^2 |{\bf k} |^2 )^2 (1- \rho_{N, {\bf k}})^2 \int_0^T
   |\widehat{ \vittest}_{\bf k}(t)|^2 dt + \, \frac{\E}{2}.
 \end{aligned}
\end{equation*}
Observe that -- for each given ${\bf k} \in {\cal T}_3^\star$ -- we have
$\rho_{N, {\bf k}} \to 1$ when $N \to \infty$. Therefore, there exists
$N_0\in\N$ (obviously depending on $\vittest$ and on $K$) such that for all
$N > N_0$,
\begin{equation*}
 \sum_{|\mathbf{k}| \le K }   (1 + \alpha^2 |{\bf k} |^2 )^2 (1-
 \rho_{N, {\bf k}})^2 
 \int_0^T |\widehat{ \vittest}_{\bf k}(t)|^2 dt <\frac{\E}{2}.
\end{equation*}
We thus obtained that
\begin{equation*}
  \forall\,\E>0\quad  \exists\,N_0=N_0(\vittest)\in\N:\quad
  \| (A-D_N) \vittest  \|_{L^2([0,T]; {\bf H}_2)}^2 < \E,\quad\forall\,N>N_0,
\end{equation*}
ending the proof. \fin

\medskip

\begin{remark}
 Let $(\vit_N, p_N) = (D_N (\wit_N), D_N (q_N))$, and define $(\vit,
 p):=(A \wit, A q)$. Our proof also shows that the field $(\vit_N,
 p_N)$ satisfies the equation
 \begin{equation}
   \label{eq:nseN}
   \begin{aligned}
     \partial_t \vit_N+ (D_N \circ G) \left ( \nabla\cdot(\vit_N\otimes
       \vit_N) \right )  - \nu\Delta \vit_N+\nabla p_N = (D_N \circ G) ({\bf f}),
     \\
     \nabla\cdot\vit_N=0, \\
     \vit_N (0, \x) = (D_N \circ G) (\vit_0) (\x). 
   \end{aligned}
 \end{equation}
 This equation is consistent with the convergence result, since $D_N
 \circ G \to {\rm Id}$, and the proof contains the fact that $(\vit,
 p)=\lim_{N\to+\infty} (\wit_N,q_N)$ is at least a distributional
 solution of the Navier-Stokes Equations~\eqref{eq:nse}. We also
 recall that the energy equality holds (see Remark~\ref{En:Approx}),
 for the solution $(\bw_N,q_N)$ of the ADM model~\eqref{eq:adm-N}.
\end{remark}
We now prove that the solution $\bw$ satisfies an ``energy
inequality.'' Observe that, from the previous estimates, we obtained
(as a consequence of the lower bound on the operator $D_N$) that
\begin{equation*}
 \begin{aligned}
   &\bw_N\in L^2([0,T]; \mathbf{H}_{2}) \quad\text{uniformly in $N$}
   \\
   &\bw_N\in L^\infty([0,T]; \mathbf{H}_{2}) \quad\text{NON uniformly
     in $N$},
 \end{aligned}
\end{equation*}
hence obtaining an estimate for $\bw$ in $L^\infty(0,T;
\mathbf{H}_{2})$ is not trivial at all since it does not derive
directly from the various estimates collected in the tables.
\begin{proposition}
  Let be given $\bu_0\in \mathbf{H}_{0}$ and $\bff\in
  L^2([0,T];\mathbf{H}_{0})$, and let $\{(\bw_N,q_N)\}_{N\in\N}$ be a
  (possibly relabelled) sequence of regular weak solutions converging
  to a weak solution $(\bw,q)$ of the filtered Navier-Stokes
  equations. Then, $\bw$ satisfies the energy inequality
 \begin{equation*}
   \frac{1}{2}    {d \over dt}\| A \wit \|^2 + 
   \nu \| \g A \wit \|^2 \le ({\bf f}, A \wit),
 \end{equation*}
 in the sense of distributions, that is 
 \begin{equation*}
   - \frac{1}{2}\int_0^T   \| A \wit(s) \|^2\phi'(s)\,ds + 
   \nu \int_0^T\| \g A \wit(s) \|^2\,\phi(s)\,ds \le \int_0^T({\bf
     f}(s), A \wit(s))\,\phi(s)\,ds, 
     \end{equation*}
     for all $\phi\in C^\infty_0(0,T)$ such that $\phi\geq0$ (From
     this one can derive the more familiar integral
     formulation~\eqref{eq:energy-final}. See~\cite{CF1988,FMRT2001a,Tem2001}). This
     implies that $\bw$ is the average of a weak (in the sense of
     Leray-Hopf) or dissipative solution $\bu$ of the Navier-Stokes
     equation~\eqref{eq:nse}. In fact, the energy inequality can also
     be read also as
 \begin{equation*}
   \frac{1}{2}{d \over  dt}\| \vit \|^2 + \nu \| \g A \vit \|^2 \le ({\bf f},  \vit). 
 \end{equation*}
 If we assume less regularity on the external force, as for instance
 $\bff\in L^2([0,T];\mathbf{H}_{-1})$, the proof remains the same and
 we obtain the corresponding inequality
 \begin{equation*}
   \frac{1}{2}{d \over  dt}\| A \wit \|^2 + 
   \nu \| \g A \wit \|^2 \le \langle{\bf f}, A \wit\rangle.
 \end{equation*}
\end{proposition}
\begin{proof}
  We proved that a (relabelled) sub-sequence $\{\bw_N\}_{n\in\N}$
  converges to a vector field $\bw$, which is solution in the sense of
  distributions of the filtered Navier-Stokes equations. The final
  step is to prove that the solution $\bw$ satisfies the ``energy
  inequality.'' In fact observe that the \textit{a priori} estimate we
  proved show that
 \begin{equation*}
   \bw\in L^2([0,T];\mathbf{H}_2)\ \Longleftrightarrow\ \bu\in
   L^2([0,T];\mathbf{H}_0),
 \end{equation*}
 and consequently the solution $\bu$ belongs to $L^2(\tore)$, for
 a.e. $t\in[0,T]$. We need now to show that the $L^2$-norm is bounded
 uniformly and that an energy balance holds. To this end we recall
 that the following energy equality (see Remark~\ref{En:Approx}),
 is satisfied, for each $N\geq0$, by the solutions of~\eqref{eq:adm-N}
 \begin{equation*}
   \frac{1}{2}{d  \over  dt}\| A^{1/2} D_N^{1/2}( \wit_N) \|^2 + \nu \| \g A^{1/2}
   D_N^{1/2}( \wit_N) \|^2 = \big(A^{-1/2} D_N^{1/2}( {\bf f}), A^{1/2} D_N^{1/2}
   (\wit_N)\big)
 \end{equation*}
 (The previous inequality has to be intended in the sense of
 distributions.)  What is needed is to pass to the limit as
 $N\to+\infty$. First, let us write the equality in integral form,
 since it will be most useful for our purposes: for all $t\in[0,T]$ 
 \begin{equation}
   \label{eq:Energy34} 
   \begin{aligned}
     &\frac{1}{2}\| A^{1/2} D_N^{1/2}( \wit_N)(t) \|^2 + \nu \int_0^t\|
     \g A^{1/2} D_N^{1/2}( \wit_N) (s)\|^2\,ds
     \\
     &\hspace{2cm} =\frac{1}{2} \| A^{1/2} D_N^{1/2} (\wit_N)(0) \|^2 +\int_0^t
     \big(A^{-1/2} D_N^{1/2}( {\bf f}), A^{1/2} D_N^{1/2}( \wit_N)\big)\,ds.
   \end{aligned}
 \end{equation}
 Observe now that, by the estimates in~(\ref{eq:Estimates}-b) we have
 established the following convergence
 \begin{equation}\label{eq:convergence-energy}
   \begin{aligned}
     & A^{1/2}D_N^{1/2}(\bw_N)\to A^{1/2} A^{1/2}(\bw) \quad\text{weakly
       in } L^2([0,T];\mathbf{H}_{1}),
     \\
     & A^{1/2}D_N^{1/2}(\bw_N)\to A^{1/2} A^{1/2}(\bw) \quad\text{weakly$\ast$
       in } L^\infty([0,T];\mathbf{H}_{0}).
  \end{aligned}
 \end{equation}
 This follows since, we are working now on a sequence $\{\bw_N\}_{N\in\N}$
 which is convergent in all the spaces previously used. In
 particular, we have that $\{A^{1/2}D_N^{1/2}(\bw_N)\}_{N\in\N}$
 converges to ``something,'' weakly in $L^2([0,T];\mathbf{H}_{1})$,
 and also $\{A^{1/2}D_N^{1/2}(\bw_N)\}_{N\in\N}$ has a
 $L^\infty([0,T];\mathbf{H}_{0})$ weak$\ast$ limit.  Finally, since we
 previously showed that
 \begin{equation*}
   \begin{aligned}
     &D_N(\bw_N)\to A\bw \qquad \text{strongly in }L^p([0,T]\times\tore),\qquad
     \forall\, p<10/3,
     \\
     &D_N(\bw_N)\to A\bw \qquad \text{weakly in }L^2([0,T];\mathbf{H}_0),
   \end{aligned}
 \end{equation*}
 the same arguments as before show also that
 \begin{equation*}
   D_N^{1/2}(\bw_N)\to A^{1/2}(\bw) \qquad \text{weakly in }L^2([0,T];\mathbf{H}_1),
 \end{equation*}
 and the uniqueness of the weak limit proves~\eqref{eq:convergence-energy}.\par
 Next, due to the assumptions on $\bff$ we have
 \begin{equation*}
   A^{-1/2}D^{1/2}_N\bff\to \bff\qquad\text{strongly in }L^2([0,T];\mathbf{H}_0).
 \end{equation*}
 In addition, since for all $N\in \N$, $\bw_N(0)=\bw(0)=\obu(0)\in
 \mathbf{H}_2$, and since the integral involving the external force $\bff$
 consists one term weakly converging, times another one strongly
 converging, it holds that
 \begin{equation*}
   \begin{aligned}
     \lim_{N\to+\infty}\bigg[\frac{1}{2} \| A^{1/2} D_N^{1/2} (\wit_N)(0)
     \|^2 +\int_0^t (A^{-1/2} D_N^{1/2}( {\bf f}), A^{1/2} D_N^{1/2}
     ( \wit_N))\,ds\bigg]
     \\
     \hspace{5cm}=\frac{1}{2} \| A\wit(0) \|^2 +\int_0^t ({\bf
       f}, A \wit)\,ds.
   \end{aligned}
 \end{equation*}
 The previous limit implies that the left-hand side
 of~\eqref{eq:Energy34} is bounded uniformly in $N\in\N$ since
 \begin{equation*}
   \begin{aligned}
     \limsup_{N\to+\infty}    \bigg[\frac{1}{2}\| A^{1/2} &D_N^{1/2}
   (\wit_N)(t) \|^2 + 
   \nu \int_0^t\|\g A^{1/2} D_N^{1/2}( \wit_N) (s)\|^2\,ds\bigg]
   \\
   & \leq\liminf_{N\to+\infty}\left[\frac{1}{2} \| A^{1/2} D_N^{1/2}(
     \wit_N)(0) \|^2 
     +\int_0^t(A^{-1/2} D_N^{1/2}( {\bf f}), A^{1/2} D_N^{1/2} (\wit_N))\,ds\right]
   \\
   & =\lim_{N\to+\infty}\left[\frac{1}{2} \| A^{1/2} D_N^{1/2}(
     \wit_N)(0) \|^2 
     +\int_0^t (A^{-1/2} D_N^{1/2} ({\bf f}), A^{1/2} D_N^{1/2}( \wit_N))\,ds\right]
   \\
   &=\frac{1}{2} \| A\wit(0) \|^2 +\int_0^t
   ({\bf f}(s), A \wit(s))\,ds.
 \end{aligned}
\end{equation*}
Next, we use the elementary inequality for the real valued sequences
$\{a_N\}_{n\in\N}$ and $\{b_N\}_{n\in\N}$
\begin{equation*}
 \limsup_{N\to+\infty} a_N+\liminf_{N\to+\infty}
 b_N\leq\limsup_{N\to+\infty} (a_N+b_N),
\end{equation*}
with
\begin{equation*}
 \begin{aligned}
   a_N:=\frac{1}{2}\| A^{1/2} D_N^{1/2} (\wit_N)(t)\|^2
   \quad\text{and}\quad b_N=\nu \int_0^t\| \g A^{1/2}
   D_N^{1/2}(\wit_N) (s)\|^2\,ds.
 \end{aligned}
\end{equation*}
(The inequality holds since we know in advance that the
right-hand side is finite.) We infer that
\begin{equation*}
 \begin{aligned}
   \limsup_{N\to+\infty} \frac{1}{2}\| A^{1/2} D_N^{1/2}
   (\wit_N)(t)\|^2
   &+ \liminf_{N\to+\infty}\nu \int_0^t\| \g A^{1/2} D_N^{1/2}(
   \wit_N) (s)\|^2\,ds
   \\
   &\leq\frac{1}{2} \| A\wit(0) \|^2 +\int_0^t ({\bf f}(s), A \wit(s))\,ds.
 \end{aligned}
\end{equation*}
By lower semi-continuity of the norm this implies that
\begin{equation*}
 \int_0^t\|  \g A \wit (s)\|^2\,ds  \leq
 \liminf_{N\to+\infty}\int_0^t\|  \g A^{1/2} D_N^{1/2} \wit_N
 (s)\|^2\,ds.
\end{equation*}
On the other hand, since $D^{1/2}\bw_N\to A^{1/2}\bw$ weakly$\ast$
$L^\infty([0,T];\mathbf{H}_0)$. Again by identification of the weak limit we
get
\begin{equation*}
 \| A\wit(t) \|^2\leq  \limsup_{N\to+\infty}   \| A^{1/2} D_N^{1/2} \wit_N(t) \|^2.
\end{equation*}
By collecting all the estimates, we have finally proved that, for all $t\in[0,T]$,
\begin{equation}\label{eq:energy-final}
 \frac{1}{2}\| A\wit(t) \|^2+\nu \int_0^t\| \g A \wit
   (s)\|^2\,ds \leq \frac{1}{2} \| A\wit(0) \|^2 +\int_0^t ({\bf
     f}(s), A \wit(s))\,ds.
\end{equation}
This can be read as the standard energy inequality for $\bu=A\bw$.
\begin{equation*}
 \begin{aligned}
   \frac{1}{2}\| \bu(t) \|^2+\nu \int_0^t\| \g \bu (s)\|^2\,ds\leq
   \frac{1}{2} \| \bu(0) \|^2 +\int_0^t ({\bf f}(s),
   \bu(s))\,ds,\qquad\forall\, t\in[0,T].
 \end{aligned}
\end{equation*}
\end{proof}
\begin{remark}
 With slightly more effort (the techniques are the usual ones as for
 the Navier-Stokes equations, see e.g.,~\cite{CF1988}) we can prove
 the same inequality, between $t_0$ and $t$, where $t_0$ is any
 Lebesgue point of the function $\|A\bw(t)\|$: For a.e. $t_0\in[0,T]$
 it holds
 \begin{equation*}
   \frac{1}{2}\| A\wit(t) \|^2+\nu  \int_{t_0}^t\|
   \g A \wit (s)\|^2\,ds\leq \frac{1}{2} \| A\wit(t_0) \|^2 +\int_{t_0}^t
   ({\bf f}(s), A \wit(s))\,ds,\quad\forall\, t\in[t_0,T].
 \end{equation*}
\end{remark}
%

\begin{thebibliography}{10}

\bibitem{AS2001}
N.~A. {A}dams and S.~Stolz.
\newblock {\em Deconvolution methods for subgrid-scale approximation in large
  eddy simulation}.
\newblock Modern Simulation Strategies for Turbulent Flow, R.T. Edwards, 2001.

\bibitem{BIL2006}
L.~C. Berselli, T.~Iliescu, and W.~J. Layton.
\newblock {\em Mathematics of {L}arge {E}ddy {S}imulation of turbulent flows}.
\newblock Scientific Computation. Springer-Verlag, Berlin, 2006.

\bibitem{CLT2006}
Y.~Cao, E.~M. Lunasin, and E.~S. Titi.
\newblock Global well-posedness of the three-dimensional viscous and inviscid
  simplified {B}ardina turbulence models.
\newblock {\em Commun. Math. Sci.}, 4(4):823--848, 2006.

\bibitem{CHOT05}
A.~{C}heskidov, D.~D. Holm, E.~Olson, and E.~S. Titi.
\newblock On a {L}eray-$\alpha$ model of turbulence.
\newblock {\em Royal Society London, Proceedings, Series A, Mathematical,
  Physical and Engineering Sciences}, 461:629--649, 2005.

\bibitem{CF1988}
P.~Constantin and C.~Foias.
\newblock {\em Navier-{S}tokes equations}.
\newblock Chicago Lectures in Mathematics. University of Chicago Press,
  Chicago, IL, 1988.

\bibitem{DG1995}
C.~R. Doering and J.~D. Gibbon.
\newblock {\em Applied analysis of the {N}avier-{S}tokes equations}.
\newblock Cambridge Texts in Applied Mathematics. Cambridge University Press,
  Cambridge, 1995.

\bibitem{DE2006}
A.~Dunca and Y.~Epshteyn.
\newblock On the {S}tolz-{A}dams deconvolution model for the large-eddy
  simulation of turbulent flows.
\newblock {\em SIAM J. Math. Anal.}, 37(6):1890--1902 (electronic), 2006.

\bibitem{FHT01}
C.~Foias, D.~D. Holm, and E.~S. Titi.
\newblock The {N}avier-{S}tokes-alpha model of fluid turbulence.
\newblock {\em Physica D}, 152:505--519, 2001.

\bibitem{FMRT2001a}
C.~Foias, O.~Manley, R.~Rosa, and R.~Temam.
\newblock {\em Navier-{S}tokes equations and turbulence}, volume~83 of {\em
  Encyclopedia of Mathematics and its Applications}.
\newblock Cambridge University Press, Cambridge, 2001.

\bibitem{Geu2003}
B.~J. Geurts.
\newblock {\em Elements of {D}irect and {L}arge {E}ddy {S}imulation}.
\newblock Edwards Publishing, Flourtown, PA, 2003.

\bibitem{La69}
O.~A. Ladyzhenskaya.
\newblock {\em The mathematical theory of viscous incompressible flow}.
\newblock Second English edition, revised and enlarged. Translated from the
  Russian by Richard A. Silverman and John Chu. Mathematics and its
  Applications, Vol. 2. Gordon and Breach Science Publishers, New York, 1969.

\bibitem{LL03}
W.~{L}ayton and {R. Lewandowski}.
\newblock A simple and stable scale similarity model for large eddy simulation:
  energy balance and existence of weak solutions.
\newblock {\em Applied Math. Letters}, 16:1205--1209, 2003.

\bibitem{LL2006a}
W.~J. Layton and R.~Lewandowski.
\newblock On a well-posed turbulence model.
\newblock {\em Discrete Contin. Dyn. Syst. Ser. B}, 6(1):111--128 (electronic),
  2006.

\bibitem{LL2006b}
W.~J. Layton and R.~Lewandowski.
\newblock Residual stress of approximate deconvolution models of turbulence.
\newblock {\em J. Turbul.}, 7:Paper 46, 21 pp. (electronic), 2006.

\bibitem{LL2008}
W.~J. Layton and R.~Lewandowski.
\newblock A high accuracy {L}eray-deconvolution model of turbulence and its
  limiting behavior.
\newblock {\em Anal. Appl. (Singap.)}, 6(1):23--49, 2008.

\bibitem{JL34}
J.~Leray.
\newblock Sur les mouvements d'une liquide visqueux emplissant l'espace.
\newblock {\em {A}cta {M}ath.}, 63:193--248, 1934.

\bibitem{LMC2005}
M.~Lesieur, O.~M\'etais, and P.~Comte.
\newblock {\em Large-eddy simulations of turbulence}.
\newblock Cambridge University Press, New York, 2005.
\newblock With a preface by James J. Riley.

\bibitem{RL09}
R.~{L}ewandowski.
\newblock On a continuous deconvolution equation for turbulence models.
\newblock {\em Lecture Notes of Ne\c cas Center for Mathematical Modeling},
  5:62--102, 2009.

\bibitem{RLbook}
R.~Lewandowski.
\newblock {\em Approximations to the {N}avier-{S}tokes {E}quations}.
\newblock In preparation, 2010.

\bibitem{Lio1969}
J.-L. Lions.
\newblock {\em Quelques m\'ethodes de r\'esolution des probl\`emes aux limites
  non lin\'eaires}.
\newblock Dunod, 1969.

\bibitem{Sag2001}
P.~Sagaut.
\newblock {\em Large eddy simulation for incompressible flows}.
\newblock Scientific Computation. Springer-Verlag, Berlin, 2001.
\newblock An introduction, With an introduction by Marcel Lesieur, Translated
  from the 1998 French original by the author.

\bibitem{SA99}
S.~Stolz and N.~A. Adams.
\newblock An approximate deconvolution procedure for large-eddy simulation.
\newblock {\em Phys. Fluids}, 11(7):1699--1701, 1999.

\bibitem{LT75}
L.~Tartar.
\newblock {\em Nonlinear partial differential equations using compactness
  method}.
\newblock report nb 1584, Mathematics Research Center, University of Wisconsin,
  Madison, 1975.

\bibitem{Tem2001}
R.~Temam.
\newblock {\em Navier-{S}tokes equations}.
\newblock AMS Chelsea Publishing, Providence, RI, 2001.
\newblock Theory and numerical analysis, Reprint of the 1984 edition.

\end{thebibliography}
\def\ocirc#1{\ifmmode\setbox0=\hbox{$#1$}\dimen0=\ht0 \advance\dimen0
  by1pt\rlap{\hbox to\wd0{\hss\raise\dimen0
  \hbox{\hskip.2em$\scriptscriptstyle\circ$}\hss}}#1\else {\accent"17 #1}\fi}
  \def\polhk#1{\setbox0=\hbox{#1}{\ooalign{\hidewidth
  \lower1.5ex\hbox{`}\hidewidth\crcr\unhbox0}}}

\end{document}